\documentclass[12pt,reqno]{amsart}
\usepackage{amsmath,amstext,amssymb,amscd}
\usepackage{verbatim}
\usepackage{enumerate}
\usepackage{mathrsfs}
\usepackage[dvipsnames,usenames]{xcolor}

\usepackage{amsthm}

\newtheorem{theorem}{Theorem}[section]
\newtheorem{lemma}[theorem]{Lemma}

\newtheorem{proposition}[theorem]{Proposition}

\theoremstyle{definition}
\newtheorem{definition}[theorem]{Definition}

\theoremstyle{remark}
\newtheorem{remark}[theorem]{Remark}

\numberwithin{equation}{section}

\begin{document}

\def\a{\alpha}
\def\b{\beta}
\def\d{\delta}
\def\g{\gamma}
\def\l{\lambda}
\def\o{\omega}
\def\s{\sigma}
\def\t{\tau}
\def\th{\theta}
\def\r{\rho}
\def\D{\Delta}
\def\G{\Gamma}
\def\O{\Omega}
\def\e{\epsilon}
\def\p{\phi}
\def\P{\Phi}
\def\S{\Psi}
\def\E{\eta}
\def\m{\mu}
\def\grad{\nabla}
\def\bar{\overline}
\newcommand{\A}{\mathcal{A}}
\newcommand{\reals}{\mathbb{R}}
\newcommand{\naturals}{\mathbb{N}}
\newcommand{\ints}{\mathbb{Z}}
\newcommand{\complex}{\mathbb{C}}
\newcommand{\rationals}{\mathbb{Q}}
\newcommand{\innerprod}[1]{\left\langle#1\right\rangle}
\newcommand{\dualprod}[1]{\left\langle#1\right\rangle}
\newcommand{\norm}[1]{\left\|#1\right\|}
\newcommand{\abs}[1]{\left|#1\right|}

\newcommand{\RR}{\mathbb{R}} 
\newcommand{\CC}{\mathbb{C}} 
\newcommand{\NN}{\mathbb{N}} 
\newcommand{\ZZ}{\mathbb{Z}} 
\newcommand{\KK}{\mathbb{K}} 
\newcommand{\QQ}{\mathbb{Q}} 
\newcommand{\TT}{\mathbb{T}} 

\newcommand{\FF}{\mathcal{F}}
\newcommand{\LL}{\mathcal{L}}

\newcommand{\eps}{\varepsilon}
\newcommand{\vrho}{\varrho}
\newcommand{\vphi}{\varphi}

\newcommand{\wkonv}[2]{#1\rightharpoonup #2} 
\newcommand{\starkonv}[2]{#1\stackrel{\ast}{\rightharpoonup} #2} 

\newcommand{\vect}[2]{\left( \begin{array}{c} #1 \\ #2 \end{array}\right)} 

\newcommand{\set}[1]{\left\lbrace#1\right\rbrace} 
\newcommand{\normhilb}[1]{\left\bracevert#1\right\bracevert} 
\newcommand{\betrag}[1]{\left\lvert#1\right\lvert} 
\newcommand{\SP}[2]{\left\langle #1, #2\right\rangle} 

\newcommand{\ve}[1]{\mathbf{#1}} 
\newcommand{\id}{\ve I} 
\renewcommand{\d}{\:\mathrm{d}} 

\title[Coupled KdV equations of Majda and Biello]
{Global Well-posedness of a System of Nonlinearly Coupled KdV equations of Majda and Biello}

\date{September 27, 2013}
\keywords{KdV equation, global well-posedness, successive time-averaging method}
\subjclass[2010]{35B34,35Q53}


\author[Y. Guo]{Yanqiu Guo}
\address[Y. Guo]
{Department of Computer Science and Applied Mathematics \\
  Weizmann Institute of Science  \\
  Rehovot 76100, Israel.} \email{yanqiu.guo@weizmann.ac.il}

\author[K. Simon]{Konrad Simon}
\address[K. Simon]
{Department of Computer Science and Applied Mathematics \\
  Weizmann Institute of Science  \\
  Rehovot 76100, Israel.} \email{konrad.simon@weizmann.ac.il}

\author[E. S. Titi]{Edriss S. Titi}
\address[E. S. Titi]
	{Department of Mathematics \\ and Department of Mechanical and
	Aerospace Engineering \\ University of California \\ Irvine,
	CA 92697-3875, USA. \\ {\bf ALSO} \\ Department of Computer
	Science and Applied Mathematics \\ Weizmann Institute of
	Science \\ Rehovot 76100, Israel.} \email{etiti@math.uci.edu
	and edriss.titi@weizmann.ac.il}

\begin{abstract}
	This paper addresses the problem of global well-posedness of a
	coupled system of Korteweg-de Vries equations, derived by
	Majda and Biello in the context of nonlinear resonant
	interaction of Rossby waves, in a periodic setting in
	homogeneous Sobolev spaces $\dot H^s$, for $s\geq  0$. Our approach is based on a successive time-averaging method developed by Babin, Ilyin and Titi \cite{B-I-T}.
\end{abstract}
\maketitle






\section{Introduction}\label{S-1}

	The present manuscript is motivated by a work of Babin, Ilyin and
	Titi~\cite{B-I-T} explaining the regularization mechanism for
	the periodic Korteweg-de Vries (KdV) equation. In \cite{B-I-T}
	the authors exploit the dispersive structure which introduces
	frequency dependent fast oscillations by means of successive
	integrations by parts and time averaging. Our aim is to adapt
	this method in order to obtain analogous well-posedness
	results for certain system of coupled Korteweg-de Vries
	equations (cKdV):
\begin{equation}
  \label{int-1}
  \begin{cases}
    A_t =\alpha A_{xxx} - (AB)_x \\
    B_t = B_{xxx} - A A_x  \\
  \end{cases}
\end{equation}
	introduced by Majda and Biello
	(see~\cite{Biel,BieMaj-3,BieMaj-1,BieMaj-2} and references therein).
	This system arises in the study of nonlinear resonant interactions of equatorial baroclinic and barotropic Rossby waves, and is a model for long range interactions between the tropical and midlatitude troposphere.
	In (\ref{int-1}), $A$ is the amplitude of an equatorially confined (baroclinic) Rossby wave packet, and $B$ is the amplitude of a (barotropic) Rossby wave packet with significant energy in the midlatitudes, and $\alpha$ is a parameter close to $1$.

Several conservation laws are known for (\ref{int-1}):
\begin{equation}
  \label{int-2}
  E_1:=\int  A \d x \:, \quad E_2:=\int  B \d x
\end{equation}
	and most important for our purpose, the total energy
\begin{equation*}
  \label{int-3}
  E:=\int A^2 + B^2 \d x \:,
\end{equation*}
	which bounds the $L^2$-norm. In addition, as elaborated in \cite{BieMaj-1}, system~(\ref{int-1}), enjoys a Hamiltonian structure, where the Hamiltonian is given by
\begin{equation*}
  \label{int-4}
  H:=\frac{1}{2} \int \alpha A_x^2 + B_x^2 +  A^2B \d x \:.
\end{equation*}
	In contrast to the 1-d KdV no more conservation laws are
	known for (\ref{int-1}), so it is not necessarily completely integrable. However, to show the global well-posedness of weak solutions of (\ref{int-1}), with initial data in $L^2$,
we only use the conservation of the total energy and in particular we will not take advantage of the conservation of the Hamiltonian.

In \cite{Biel}, Biello used the change of variables
	$U=\frac{1}{\sqrt{2}}(\sqrt{2}B+A)$ and
	$V=\frac{1}{\sqrt{2}}(\sqrt{2}B-A)$ to transform (\ref{int-1})  into an idealized
``symmetric" model (when $\alpha=1$)
\begin{equation}
  \label{int-1a}
  \begin{cases}
    U_t=U_{xxx}-U U_x+\frac{1}{2}(UV)_x \\
    V_t=V_{xxx}-V V_x+\frac{1}{2} (UV)_x
  \end{cases}
\end{equation}
and studied its soliton solutions.
	In this paper, we will consider system~(\ref{int-1a}) subject to periodic boundary conditions, with basic periodic domain $\TT=[0,2\pi]$; which is equivalent to consider (\ref{int-1a}) on the unit circle. Note
that there are two invariant subspaces, $U=0$ or $V=0$. In case
of $U=0$ the solution for $V$ evolves according to a standard KdV (respectively $U$ if $V=0$ is taken).

Since system~(\ref{int-1a}) is closely connected to the KdV equation $u_t=u_{xxx}+uu_x$, we now briefly review some important results concerning the KdV, with periodic boundary condition. In his seminal
	papers~\cite{Bou-1,Bou-2} Bourgain introduced a new type of
	weighted Sobolev spaces $X^{s,b}(\RR \times \TT)$ for functions in time and
    space, the so-called dispersive Sobolev spaces, which is the closure of the Schwartz space under the norm
\begin{equation}
  \label{int-5}
\norm{u}_{X^{s,b}(\RR \times \TT)} = \norm{\langle k \rangle^s\langle \tau + k^3  \rangle^b\widehat u (\tau,k)}_{L_\tau^2 l_k^2(\RR \times \ZZ)} \:,
\end{equation}
where $\langle\cdot\rangle=(1+|\cdot|^2)^{1/2}$ and
$\widehat u$ denotes the Fourier transform in space and time. These spaces reflect the fact that the Fourier transform of a solution of the
	unperturbed (dispersive) part of the KdV is supported on the
	characteristic hyperplane $\tau+k^3=0$ described by its dispersion
	relation. In fact, the $X^{s,b}$ spaces are an efficient tool to capture the phenomenon that the solutions to KdV, after localisation in time, have space-time Fourier transform supported near the characteristic surface; thus the nonlinearity does not significantly alter the space-time Fourier ``path" of the solution, at least for short time (see \cite{Tao}).
    The definition can of course be adapted to account
	for other dispersive PDEs like Schr\"{o}dinger equation. Using
	this Bourgain proved local well-posedness the KdV in $L^2(\TT)$ by
	means of Banach's Fixed Point principle. This result was
	improved by Kenig, Ponce and Vega~\cite{K-P-V}. They proved a
	sharp bilinear estimate for the norm in~(\ref{int-5}) and
	showed local well-posedness in $H^{-1/2}(\TT)$. The
	corresponding global well-posedness result in $H^{-1/2}(\TT)$
	has been proved by Colliander, Keel, Staffilani, Takaoka and
	Tao~\cite{CKSTT-1} by employing the $I$-method (or the method of almost conserved quantities); here, ``$I$" stands for a mollification operator, acting like the {\bf I}dentity on low frequencies, and like an {\bf I}ntegration operator on high frequencies. Kappeler
	and Topalov \cite{Kappeler-Topalov} were able to prove global well-posedness in
	$H^{-1}(\TT)$ by using the complete integrability of the KdV.

	The Majda-Biello system~(\ref{int-1}) is a member of a wider
	class of KdV-type systems. Another model among
	several other systems of this class is for
	example the Gear-Grimshaw system~\cite{GearGrimshaw}. In~\cite{Oh-1} Oh investigated
	system~(\ref{int-1}) by employing $X^{s,b}$-estimates and
	obtained local well-posedness results depending on the
	value of the parameter $\alpha$. For $\alpha=1$ he obtained
	local well-posedness in (a cross-product of) $H^{-1/2}(\TT)$
	and he proved local well-posedness for almost every $\alpha\in
	(0,1)$ in $H^s(\TT)$, $s>1/2$. The reason for this is that if $\alpha\not=0$
	certain nontrivial resonances occur (which can be described by
	using diophantine conditions) because the space-time Fourier transforms of solutions of the two
	linear parts of the system are supported on different
	hyperplanes described by their dispersion
	relation. Corresponding global well-posedness results have also
	been proved by Oh~\cite{Oh-2} using the $I$-method. If
	$\alpha=1$ system~(\ref{int-1}) is globally well-posed in
	$H^{-1/2}(\TT)$; while for almost every
	$\alpha\in (0,1)$ it is globally well-posed in $H^s(\TT)$, $s>5/7$.

        	In this paper we use the technique of successive
        	differentiation by parts introduced by Babin, Ilyin
        	and Titi~\cite{B-I-T} on system~(\ref{int-1a}) with
            periodic boundary condition. The first step of the method is to apply
        	the transform
\begin{equation}
  \label{int-6}
  U_k(t)=e^{-ik^3t}u_k(t) \:,\;   V_k(t)=e^{-ik^3t}v_k(t) \:, \; k\in \ZZ,
\end{equation}
	on the Fourier coefficients. This transform represents the
	action of the unitary group generated by the third derivative,
	$\Psi(t)=e^{\partial_x^3t}$, on each Fourier coefficient. In
	the terminology of quantum mechanics transform~(\ref{int-6})
	means the transition to the so-called interaction
	representation \cite{Ginibre}. This can be interpreted in terms of the spaces
	$X^{s,b}$: a function $u$ of space and time is in $X^{s,b}$ if
	and only if its interaction representation $\Psi(-t) u(t)$ is in the mixed
	Sobolev space $H_t^bH_x^s$. The
	transform~(\ref{int-6}) generates a fast rotation term into the equation, and
    then several forms of the system are derived
	using successive differentiations by parts in time (which
	correspond to integrations by parts in time) after resonances
	are singled out. The equation becomes of higher algebraic order but we
	can take advantage of smoothing properties of the higher order
	operators involved which allows less regular solutions. In principle,
    this is similar to the idea of normal forms by Shatah \cite{Sha}.
After establishing the global existence in the homogeneous Sobolev space
	$\dot H^s$ for $s>0$ by using Galerkin method, we prove uniqueness of solutions by means of the Banach's Fixed Point Theorem. In \cite{B-I-T}, for constructing a strict contraction mapping, the authors inverted a linear operator that involves the initial value. The inversion and
	a time-independent estimate on its inverse were done by finding
	an explicit solution to a boundary value problem for an
	ODE. However, for the Majda-Biello cKdV system, we
	run into the difficulty of now having to solve a system of
	1D boundary-value problem explicitly. For the sake of bypassing this obstacle, we use a proper splitting of solutions based on high and low Fourier modes, and recast the differentiation by parts procedures to terms involving high frequencies. This idea avoids treating the invertibility of a linear operator, and simply takes advantage of the time-averaging induced squeezing. Such strategy was first introduced in \cite{B-I-T} to deal with the less regular initial data in $H^s(\mathbb T)$ for $s\in [0,1/2]$. The authors of \cite{KwoOh} followed the idea from \cite{B-I-T} to obtain unconditional well-posedness of modified KdV in $H^s(\mathbb T)$ for $s\geq 1/2$. Finally, we must stress that, the present work does not aim to improve the results in \cite{Oh-1,Oh-2}; but, our purpose is to provide another example of employing the techniques in \cite{B-I-T} which is general enough to apply to other nonlinear dispersive and wave equations for establishing global well-posedness. Although the $X^{s,b}$ spaces are a powerful tool to study dispersive equations, in this paper, we simply use the standard Sobolev spaces $H^s$ in a systematic and natural manner.

	This paper is organized as follows. In
	Sections~\ref{S-2},~\ref{first-dbp} and~\ref{sec-dbp} we
	derive several forms of the cKdV (\ref{int-1a}) analogously to~\cite{B-I-T}. In
	Section~\ref{S-4} we prove global existence of a solution in the homogeneous Sobolev space
	$\dot H^s$, for $s>0$, using a Galerkin
	scheme and we establish uniform bounds for the solution on
	each finite time interval. Section~\ref{regular} is dedicated to
	regular initial data, that is $s>1/2$, where uniqueness is obtained by
	means of Banach's Contraction Principle.
	Section~\ref{irregular} addresses less regular initial data, i.e.,
	$s\in [0,1/2]$. For the sake of convenience we use
	similar notations as in~\cite{B-I-T} due to the fact that
	most of the nonlinear operators occurring in this work have
	the same mapping properties as the ones proven
	there. Therefore, throughout this work, most relevant estimates
	of nonlinear operators will be taken from the appendix section
	and their proofs can be found in~\cite{B-I-T}.

\smallskip

\section{Transformations of the system and main results}\label{S-2}
In this section, we write the cKdV (\ref{int-1a}) in terms of Fourier coefficients, and use a transform of variables in order to introduce oscillating exponentials into the nonlinear term. Based on the transformed system, we shall define a notion of (weak) solutions and state the main results of the present paper.

As mentioned in the Introduction, we consider the Majda-Biello system
\begin{equation}
  \label{eq:1}
  \begin{cases}
    U_t=U_{xxx}-U U_x+\frac{1}{2}(UV)_x \\
    V_t=V_{xxx}-V V_x+\frac{1}{2} (UV)_x \\
    U(0,x)=U^{\text{in}}(x)\:, \; V(0,x)=V^{\text{in}}(x)
  \end{cases}
\end{equation}
	where $x\in\TT=[0,2\pi]$, with periodic boundary condition $U(t,0)=U(t,2\pi)$ (the same
	for $V$). Here, $U$ and $V$ are real-valued functions. If $(U,V)$ is a smooth solution of~(\ref{eq:1}) we observe, from (\ref{int-2}),
	the conservation of the mean values, i.e.,  $$\frac{d}{dt} \int_0^{2\pi} U(t,x) dx =\frac{d}{dt} \int_0^{2\pi} V(t,x) dx=0.$$
We assume from now that the initial data and the solution both have spatial mean value zero.

Denote $\ZZ_0:=\ZZ \backslash \{0\}$. We make a Fourier expansion for $U$
\begin{equation}
  \label{eq:3}
  U(t,x)=\sum_{k\in\ZZ_0}U_k(t)e^{ikx} \:,\; U_k\in\CC \:,\; U_k(t)=\frac{1}{2\pi}\int_0^{2\pi}U(t,x)e^{-ikx}\d x \:,
  \quad k\in\ZZ_0\:,
\end{equation}
	as well as for $V$. Furthermore, we observe that
	$\overline U_k=U_{-k}$, since we are seeking real valued
	solutions. Therefore, we denote by $\dot H^s(\TT)$ the homogeneous Sobolev spaces of order $s$ on $\TT$, which is a subspace of $L^1(\TT)$ functions with mean value zero endowed with the norm
\begin{equation*}
  \label{eq:4}
  \norm{U}_{\dot H^s}^2:=\sum_{k\in\ZZ_0}|k|^{2s}|U_k|^2 \:.
\end{equation*}
	For $s=0$ this is a normalized version of the $L^2$-Norm
\begin{equation*}
  \norm{U}_{\dot H^0}^2=\frac{1}{2\pi}\norm{U}_{L^2}^2 \:.
\end{equation*}

	Plugging~(\ref{eq:3}) into equation~(\ref{eq:1}) yields the
	infinite coupled system
\begin{equation}
  \label{eq:6}
  \begin{cases}
    \partial_t U_k = -ik^3U_k + \frac{1}{2}ik\sum_{k_1+k_2=k}\left(U_{k_1}V_{k_2}-U_{k_1}U_{k_2}\right) \\
    \partial_t V_k = -ik^3V_k + \frac{1}{2}ik\sum_{k_1+k_2=k}\left(U_{k_1}V_{k_2}-V_{k_1}V_{k_2}\right) \\
    U_k(0)=U_k^{\text{in}} \:,\; V_k(0)=V_k^{\text{in}} \:,
  \end{cases}
\end{equation}
	for $k\in\ZZ_0$. We now apply the transform
\begin{equation}
  \label{eq:7}
  U_k(t)=e^{-ik^3t}u_k(t) \:,\;   V_k(t)=e^{-ik^3t}v_k(t) \:, \; k\in\ZZ_0 \:,
\end{equation}
	in order to eliminate the linear terms in~(\ref{eq:6}). By means of the identity
\begin{equation}
  \label{eq:8}
  (k_1+k_2)^3 = 3(k_1+k_2)k_1k_2 + k_1^3 + k_2^3
\end{equation}
	equation~(\ref{eq:6}) becomes
\begin{equation}
  \label{eq:9}
  \begin{cases}
    \partial_t u_k = \frac{1}{2}ik\sum_{k_1+k_2=k}e^{3ikk_1k_2t}\left(u_{k_1}v_{k_2}-u_{k_1}u_{k_2}\right) \\
    \partial_t v_k = \frac{1}{2}ik\sum_{k_1+k_2=k}e^{3ikk_1k_2t}\left(u_{k_1}v_{k_2}-v_{k_1}v_{k_2}\right) \\
    u_k(0)=u_k^{\text{in}} \:,\; v_k(0)=v_k^{\text{in}}
  \end{cases}, \;\;k\in \ZZ_0\;.
\end{equation}
We emphasize that the fast oscillating term $e^{3ikk_1k_2 t}$ in (\ref{eq:9}) reduces the ``strength" of the nonlinear term and make it milder, which is the underlying mechanism for prolonging the lifespan of the solutions \cite{B-I-T}.

	Observe that identity~(\ref{eq:8}) was also used in the
	original work of Bourgain~\cite{Bou-2}.
	Also notice transform~(\ref{eq:7}) is isometric in $\dot
	H^s$. Using the same notation as in~\cite{B-I-T}, we can write~(\ref{eq:9}) as
\begin{equation}
  \label{eq:10}
  \partial_t \vect{u_k}{v_k}=\vect{B_1(u,v)_k-B_1(u,u)_k}{B_1(u,v)_k-B_1(v,v)_k} \:,\; k\in\ZZ_0 \:,
\end{equation}
with $(u(0),v(0))=(u^{\text{in}},v^{\text{in}})$, where the bilinear operator $B_1(\phi,\psi)$ is defined by
\begin{equation}
  \label{eq:10a}
  B_1(\phi,\psi)_k:=\frac{1}{2}ik\sum_{k_1+k_2=k}e^{3ikk_1k_2t}\phi_{k_1}\psi_{k_2} \:.
\end{equation}

We now define our notion of a (weak) solution of~(\ref{eq:9}).
\begin{definition} \label{dfn-solution}
 Given the initial data $(u^{\text{in}},v^{\text{in}})\in (\dot H^0)^2$. We call a function $(u,v)$ a \emph{solution} of~(\ref{eq:9}) over the time interval $[0,T]$ if $(u,v)\in L^{\infty}([0,T];(\dot H^0)^2)$ and if the integrated version of~(\ref{eq:9}), that is
\begin{equation}
    \label{eq:11}
    \vect{u_k(t)}{v_k(t)}-\vect{u_k^{\text{in}}}{v_k^{\text{in}}}=\frac{1}{2}ik
    \int_0^t\sum_{k_1+k_2=k}e^{3ikk_1k_2\t}\vect{u_{k_1}v_{k_2}-u_{k_1}u_{k_2}}{u_{k_1}v_{k_2}-v_{k_1}v_{k_2}}\d\tau \:,
\end{equation}
	is satisfied for every $k\in\ZZ_0$. By means of~(\ref{eq:7})
	we ultimately get a (weak) solution of the Majda-Biello system (\ref{eq:1}).
\end{definition}

\begin{remark} \label{rmk-abscont}
It is readily seen from the Cauchy-Schwarz inequality that
\begin{equation}
    \label{eq:11a}
    \sup_{t\in [0,T]}\left|\sum_{k_1+k_2=k}e^{3ik k_1 k_2 t}(u_{k_1}(t)v_{k_2}(t)-u_{k_1}(t)u_{k_2}(t))\right|
    \leq  \norm{(u,v)}_{L^\infty([0,T];(\dot H^0)^2)}^2 \:.
\end{equation}
	Therefore, by (\ref{eq:11}), $u_k(t)$ is absolutely continuous (for every $k$)
	over the interval $[0,T]$, which implies $u_k(t)$ is differentiable a.e. on $[0,T]$ (same for $v_k(t)$). Consequently, equation~(\ref{eq:9}),
	the differential form of~(\ref{eq:11}), is satisfied for almost
	every $t\in [0,T]$. Also the continuity of $u_k(t)$ and $v_k(t)$ implies that $(u,v)$ is a weak continuous function mapping from $[0,T]$ to $(\dot H^0)^2$.
\end{remark}

The main result of this manuscript is the global well-posedness for the cKdV (\ref{eq:9}) in the space  $(\dot H^{s})^2$, for $s\geq 0$. More precisely, we have the following:

\begin{theorem}\label{exist}
{\bf(Global well-posedness)}
  	Let $s\geq 0$, $(u^{\emph{in}},v^{\emph{in}})\in (\dot H^{s})^2$, and $T>0$. Then there exists a unique solution $(u(t),v(t))\in C([0,T];(\dot H^s)^2)$ of the cKdV (\ref{eq:9}), in the sense of Definition \ref{dfn-solution},  with $(u(0),v(0))=(u^{\emph{in}},v^{\emph{in}})$, such that
\begin{equation}
\label{gal31}
\norm{(u,v)}_{L^\infty([0,T];(\dot
H^{s})^2)} \leq   C\left(\norm{(u^{\emph{in}},v^{\emph{in}})}_{(\dot H^0)^2},T,s\right)\:.
\end{equation}
	Moreover, the quantity
\begin{align} \label{conserve}
\mathcal E(u(t),v(t)):= 2\norm{u(t)}_{\dot H^0}^2+2\norm{v(t)}_{\dot H^0}^2+\norm{u(t)-v(t)}_{\dot H^0}^2
\end{align}
is conserved in time. In addition, the solution depends continuously on the initial data in the sense that
\begin{align*}
\norm{(u,v)-(\tilde u,\tilde v)}_{L^{\infty}([0,T];(\dot H^s)^2)} \leq L\norm{(u^{\emph{in}},v^{\emph{in}})-(\tilde u^{\emph{in}},\tilde v^{\emph{in}})}_{(\dot H^s)^2},
\end{align*}
where $(u,v)$, $(\tilde u,\tilde v)$ are solutions of (\ref{eq:9}) corresponding to the initial data $(u^{\emph{in}},v^{\emph{in}})$, $(\tilde u^{\emph{in}},\tilde v^{\emph{in}})$ respectively,
and $L>0$ depends on $T$, $s$, and $\max\left\{\norm{(u^{\emph{in}},v^{\emph{in}})}_{(\dot H^0)^2}, \norm{(\tilde u^{\emph{in}},\tilde v^{\emph{in}})}_{(\dot H^0)^2}\right\}$.
\end{theorem}

\smallskip

\section{First differentiation by parts in time}\label{first-dbp}

	As already mentioned above we want to derive different forms
	of the cKdV~(\ref{eq:9}) in order to
	obtain operators which have better mapping properties
	than $B_1$, given in (\ref{eq:10a}), and whose regularity is specified in Lemma~\ref{L:B1}. This will be formally done by the
	differentiation by parts procedure as described in the
	sequel. One observes that (\ref{eq:9}) is equivalent to
\begin{align} \label{1st-dif-1}
  \partial_t\left[\vect{u_k}{v_k}-\frac{1}{6}\sum_{k_1+k_2=k}\frac{e^{3ikk_1k_2t}}{k_1k_2}
    \vect{u_{k_1}v_{k_2}-u_{k_1}u_{k_2}}{u_{k_1}v_{k_2}-v_{k_1}v_{k_2}}\right] \notag\\
  =-\frac{1}{6}\sum_{k_1+k_2=k}\frac{e^{3ikk_1k_2t}}{k_1k_2}
  \partial_t\vect{u_{k_1}v_{k_2}-u_{k_1}u_{k_2}}{u_{k_1}v_{k_2}-v_{k_1}v_{k_2}}.
\end{align}

	Notice that since we assume the spatial means are zero,
	the indices $k,k_1,k_2$ in the above expressions are never
	equal to zero. That is, there is no resonance between the
	nonlinearity of the cKdV system and the linear operator $\partial_x^3$.

We look at a typical term on the right-hand side of (\ref{1st-dif-1}). By using (\ref{eq:9}) we deduce
\begin{align*}
&\sum_{k_1+k_2=k}\frac{e^{3ikk_1k_2 t}}{k_1 k_2} \partial_t(u_{k_1} v_{k_2})
=\sum_{k_1+k_2=k}\frac{e^{3ikk_1k_2 t}}{k_1 k_2}
(u_{k_1} \partial_t v_{k_2} + v_{k_2} \partial_t u_{k_1})\notag\\
&=\sum_{k_1+k_2=k}\frac{e^{3ikk_1k_2 t}}{k_1 k_2}
(u_{k_1} \partial_t v_{k_2} + v_{k_1} \partial_t u_{k_2})\notag\\
&=\sum_{k_1+k_2=k}\frac{e^{3ikk_1k_2 t}}{k_1 k_2}\left(\frac{i}{2} k_2
\sum_{\alpha+\beta=k_2}e^{3ik_2 \alpha \beta t}[u_{k_1}(u_{\alpha}v_{\beta}-v_{\alpha}v_{\beta})
+v_{k_1}(u_{\alpha}v_{\beta}-u_{\alpha}u_{\beta})]\right) \notag\\
&=\frac{i}{2}\sum_{k_1+k_2+k_3=k}\frac{e^{3i(k_1+k_2)(k_2+k_3)(k_1+k_3)t}}{k_1}
(u_{k_1}u_{k_2}v_{k_3}-u_{k_1}v_{k_2}v_{k_3} \notag\\
&\hspace{3 in}+v_{k_1}u_{k_2}v_{k_3}-v_{k_1}u_{k_2}u_{k_3}).
\end{align*}
In the same manner, we can manipulate every term on the right-hand side of (\ref{1st-dif-1}) to arrive at the
\emph{first form of the cKdV}:
\begin{align}
  \label{eq:13}
 \partial_t\left[
    \vect{u_k}{v_k}-\vect{B_2(u,v)_k-B_2(u,u)_k}{B_2(u,v)_k-B_2(v,v)_k}
  \right] =\ve{R}_3(u,v)_k \:,
\end{align}
	where the bilinear operator $B_2(\phi,\psi)$ is defined by
\begin{equation}
  \label{eq:13a}
  B_2(\phi,\psi)_k:=\frac{1}{6}\sum_{k=k_1+k_2}\frac{e^{3ikk_1k_2t}}{k_1k_2}\phi_{k_1}\psi_{k_2} \:,
\end{equation}
	and all terms in every component of $\ve{R}_3(u,v)$ have the structure $\pm \frac{i}{12}R_3(\phi,\psi,\xi)$, where
\begin{equation}
  \label{eq:13b}
  R_3(\phi,\psi,\xi):=\sum_{k=k_1+k_2+k_3}\frac{e^{3i(k_1+k_2)(k_2+k_3)(k_1+k_3)t}}{k_1}
  \phi_{k_1}\psi_{k_2}\xi_{k_3} \:,
\end{equation}
where each of $\phi$, $\psi$, $\xi$ may be either $u$ or $v$. For the sake of conciseness, we do not provide the exact formula of $\ve{R}_3(u,v)$.

\begin{remark}
The mapping properties of $B_2$ and $R_3$ are better than those of $B_1$ (see the Appendix). So the first form (\ref{eq:13}) is ``milder" than the original cKdV (\ref{eq:9}), which is the purpose of the differentiation by parts procedure.
On the other hand, we remark that these two forms, (\ref{eq:13}) and  (\ref{eq:9}), are not equivalent. Clearly, any smooth functions that satisfy the original equation (\ref{eq:9}) are also solutions of the newly derived equation (\ref{eq:13}), but the converse may not be true. Nonetheless, if one is able to show the uniqueness of solutions to (\ref{eq:13}), then it follows that (\ref{eq:9}) cannot have more than one solution. Hence, in order to prove the uniqueness for (\ref{eq:9}), our strategy is to consider the equation after the first (or the second) differentiation by parts procedure (see Section \ref{regular} and \ref{irregular} for details).
\end{remark}

\smallskip

\section{Second differentiation by parts in time}\label{sec-dbp}

	In order to establish a priori estimates for higher order
	Sobolev norms than $\dot H^0$, namely in $\dot H^s$,
    for $s>0$, we
	cannot use the operator $R_3$ due to its restricted regularity
	properties. Therefore, we need to perform a second
	differentiation by parts in time. But before doing this we
	must care for the nonlinear resonances which reveal themselves as
	obstacles for this procedure. Our
	aim is to decompose $\ve R_3(u,v)$ into a sum of two parts:
\begin{align} \label{split}
\ve R_3(u,v)=\ve R_{\text{3res}}(u,v)+\ve R_{\text{3nres}}(u,v),
\end{align}
	where the first part $\ve {R}_{\text{3res}}(u,v)$
    involves the resonances  and the second part $\ve {R}_{\text{3nres}}(u,v)$
	is suitable for the differentiation by parts procedure
	(non-resonance part). Recall every term in $\ve R_3(u,v)$ has the structure of   $\pm \frac{i}{12}R_3(\phi,\psi,\xi)$ defined in (\ref{eq:13b}) where each of $\phi$, $\psi$, $\xi$ may be either $u$ or $v$.
Thus, the decomposition of $R_3(\phi,\psi,\xi)_k=R_{\text{3res}}(\phi,\psi,\xi)_k+R_{\text{3nres}}(\phi,\psi,\xi)_k$ describes the split (\ref{split}). Indeed, we set
\begin{align} \label{def-res}
  R_{\text{3res}}(\phi,\psi,\xi)_k
  :=\sum_{k_1+k_2+k_3=k}^{\text{res}} \frac{\phi_{k_1}\psi_{k_2}\xi_{k_3}}{k_1} \:, \;\;k\in \ZZ_0\:,
\end{align}
where the summation is carried out over the set of subscripts $k_1$, $k_2$ and $k_3$ satisfying
$(k_1+k_2)(k_2+k_3)(k_1+k_3)=0$ (the resonances). Also, we denote
\begin{align} \label{def-nres}
  R_{\text{3nres}}(\phi,\psi,\xi)_k
  :=\sum_{k_1+k_2+k_3=k}^{\text{nonres}}\frac{e^{3i(k_1+k_2)(k_2+k_3)(k_1+k_3)t}}{k_1}\phi_{k_1}\psi_{k_2}\xi_{k_3} \:,\;\;k\in \ZZ_0\:,
\end{align}
where the sum is taken over all $k_1$, $k_2$ and $k_3$
such that $(k_1+k_2)(k_2+k_3)(k_1+k_3)\not=0$ (the non-resonances).

Let us first consider the resonances. Same as \cite{B-I-T}, the set of subscripts $k_1$, $k_2$ and $k_3$ satisfying $(k_1+k_2)(k_2+k_3)(k_1+k_3)=0$ and $k_1+k_2+k_3=k\in \ZZ_0$ is the union of six disjoint sets $S_1,\ldots,S_6$:
\begin{align*}
&S_1=\{k_1+k_2=0\}\cap \{k_2+k_3=0\}\Leftrightarrow k_1=k, \;k_2=-k, \; k_3=k\:;\notag\\
&S_2=\{k_1+k_2=0\}\cap \{k_3+k_1=0\}\Leftrightarrow k_1=k, \; k_2=-k, \; k_3=-k\:;\notag\\
&S_3=\{k_2+k_3=0\}\cap \{k_3+k_1=0\}\Leftrightarrow k_1=k, \; k_2=k, \; k_3=-k\:;\notag\\
&S_4=\{k_1+k_2=0\}\cap \{k_2+k_3\not=0\} \cap \{k_3+k_1\not=0\} \Leftrightarrow \notag\\
&\hspace{0.4 in} k_1=j,\; k_2=-j, \; k_3=k,\;  |j|\not= |k| \:; \notag\\
&S_5=\{k_2+k_3=0\}\cap \{k_1+k_2\not=0\} \cap \{k_3+k_1\not=0\} \Leftrightarrow \notag\\
&\hspace{0.4 in} k_1=k,\; k_2=j, \; k_3=-j,\;  |j|\not= |k| \:; \notag\\
&S_6=\{k_3+k_1=0\}\cap \{k_1+k_2\not=0\} \cap \{k_2+k_3\not=0\} \Leftrightarrow \notag\\
&\hspace{0.4 in} k_1=j,\; k_2=k, \; k_3=-j,\;  |j|\not= |k| \:,
\end{align*}
where $j\in \ZZ_0$. As a result,
\begin{align*}
R_{\text{3res}}(\phi,\psi,\xi)_k &=  \sum_{m=1}^6 \sum_{S_m} \frac{\phi_{k_1}\psi_{k_2}\xi_{k_3}}{k_1}\notag\\
&=\frac{\phi_k\psi_{-k}\xi_k}{k}+\frac{\phi_k\psi_{-k}\xi_{-k}}{k}+\frac{\phi_k\psi_{k}\xi_{-k}}{k}
+\xi_k \sum_{j\in \ZZ_0, |j|\not=|k|}\frac{\phi_{j}\psi_{-j}}{j} \notag\\
&\hspace{0.3 in}+\frac{\phi_k}{k} \sum_{j\in \ZZ_0, |j|\not=|k|} \psi_{j}\xi_{-j}+\psi_k \sum_{j\in \ZZ_0, |j|\not=|k|}\frac{\phi_{j}\xi_{-j}}{j}.
\end{align*}
Consequently, we deduce the following mapping property of
$R_{\text{3res}}: \dot H^{s-1}\times \dot H^s \times \dot H^s \rightarrow \dot H^s$, $s\geq 0$,
\begin{align*}
\norm{R_{\text{3res}}(\phi,\psi,\xi)}_{\dot H^s}
\leq &C \big(\norm{\phi}_{\dot H^{s-1}} \norm{\psi}_{\dot H^0} \norm{\xi}_{\dot H^0}+\norm{\xi}_{\dot H^s} \norm{\phi}_{\dot H^{-1}} \norm{\psi}_{\dot H^0}  \notag\\
&+ \norm{\phi}_{\dot H^{s-1}} \norm{\psi}_{\dot H^0} \norm{\xi}_{\dot H^0}+\norm{\psi}_{\dot H^s} \norm{\phi}_{\dot H^{-1}} \norm{\xi}_{\dot H^0} \big).
\end{align*}
 Notice that $1/k_1$ in the definition (\ref{def-res}) of $R_{\text{3res}}(\phi, \psi, \xi)$
 has smoothing effect on the variable $\phi$.
 Since every term in $\ve R_{\text{3res}}(u,v)$ has the structure of $\pm \frac{i}{12}R_{\text{3res}}(\phi,\psi,\xi)$ where $\phi$, $\psi$, $\xi$ may be either $u$ or $v$, it follows that
\begin{align} \label{resbound}
\norm{\ve R_{\text{3res}}(u,v)}_{(\dot H^s)^2}\leq C \norm{(u,v)}_{(\dot H^0)^2}^2 \norm{(u,v)}_{(\dot H^s)^2}, \text{\;\;for\;\;} s\geq 0.
\end{align}

Next we perform differentiation by parts to the non-resonance part $\ve R_{\text{3nres}}(u,v)$. Since every term in $\ve R_{\text{3nres}}(u,v)$ is in the form of $\pm \frac{i}{12}R_{\text{3nres}}(\phi,\psi,\xi)$ where $\phi$, $\psi$, $\xi$ may be either $u$ or $v$, it is sufficient to work on a typical term $R_{\text{3nres}}(u,u,v)$ in order to demonstrate our strategy. Observe that
\begin{align}
  \label{dbp9}
  &R_{\text{3nres}}(u,u,v)_k  \notag\\
  &=\sum_{k_1+k_2+k_3=k}^{\text{nonres}}\frac{e^{3i(k_1+k_2)(k_2+k_3)(k_1+k_3)t}}{k_1}
  u_{k_1}u_{k_2}v_{k_3}  \notag\\
  &= \frac{1}{3i}\partial_t B_3(u,u,v)_k-\frac{1}{3i}\sum_{k_1+k_2+k_3=k}^{\text{nonres}}\frac{e^{3i(k_1+k_2)(k_2+k_3)(k_1+k_3)t}}{k_1(k_1+k_2)(k_2+k_3)(k_1+k_3)}\notag\\
 &\hspace{1.5 in}\times (\partial_t u_{k_1}u_{k_2}v_{k_3}+u_{k_1}\partial_t u_{k_2}v_{k_3}+u_{k_1}u_{k_2}\partial_t v_{k_3}) \:,
\end{align}
	where the trilinear operator $B_3(\phi,\psi,\xi)$ is defined by
\begin{equation}
  \label{dbp10}
  B_3(\phi,\psi,\xi)_k:=
\sum_{k_1+k_2+k_3=k}^{\text{nonres}}
\frac{e^{3i(k_1+k_2)(k_2+k_3)(k_1+k_3)t}}{k_1(k_1+k_2)(k_2+k_3)(k_1+k_3)}
\phi_{k_1}\psi_{k_2}\xi_{k_3} \:.
\end{equation}
	Using equation~(\ref{eq:9}) we observe
\begin{align} \label{star}
  &\sum_{k_1+k_2+k_3=k}^{\text{nonres}}\frac{e^{3i(k_1+k_2)(k_2+k_3)(k_1+k_3)t}}{k_1(k_1+k_2)(k_2+k_3)(k_1+k_3)}
  (\partial_t u_{k_1}u_{k_2}v_{k_3}+u_{k_1}\partial_t u_{k_2}v_{k_3}+u_{k_1}u_{k_2}\partial_t v_{k_3})  \notag\\
  &=\frac{i}{2}\Bigg[
    \sum_{k_1+k_2+k_3+k_4=k}^{\text{nonres}}\frac{e^{i\Phi(\ve k)t}}{(k_1+k_2)(k_1+k_3+k_4)(k_2+k_3+k_4)}
    (v_{k_1}u_{k_2}u_{k_3}v_{k_4}-v_{k_1}u_{k_2}u_{k_3}u_{k_4})  \notag\\
  &+\sum_{k_1+k_2+k_3+k_4=k}^{\text{nonres}}\frac{e^{i\Phi(\ve k)t}(k_3+k_4)}{k_1(k_1+k_2)(k_1+k_3+k_4)(k_2+k_3+k_4)}
  (u_{k_1}v_{k_2}u_{k_3}v_{k_4}-u_{k_1}v_{k_2}u_{k_3}u_{k_4}) \notag\\
  &+\sum_{k_1+k_2+k_3+k_4=k}^{\text{nonres}}\frac{e^{i\Phi(\ve k)t}(k_3+k_4)}{k_1(k_1+k_2)(k_1+k_3+k_4)(k_2+k_3+k_4)}
    (u_{k_1}u_{k_2}u_{k_3}v_{k_4}-u_{k_1}u_{k_2}v_{k_3}v_{k_4})
  \Bigg],
\end{align}
	where $\ve k:=(k_1,k_2,k_3,k_4)$ and $\Phi(\ve
	k):=(k_1+k_2+k_3+k_4)^3-k_1^3-k_2^3-k_3^3-k_4^3$. However, the
    exact expression of the
	phase function $\Phi(\ve k)$ is not important in our
	case. Proceeding in the same manner for each non-resonance term, we obtain a sum of expressions in the structure
\begin{equation*}
  \label{dbp13}
  B_4^1(\phi,\psi,\xi,\eta)_k=
  \sum_{k_1+k_2+k_3+k_4=k}^{\text{nonres}}\frac{e^{i\Phi(\ve k)t}}{(k_1+k_2)(k_1+k_3+k_4)(k_2+k_3+k_4)}\phi_{k_1}\psi_{k_2}\xi_{k_3}\eta_{k_4}
\end{equation*}
or
\begin{equation*}
  \label{dbp14}
  B_4^2(\phi,\psi,\xi,\eta)_k=
  \sum_{k_1+k_2+k_3+k_4=k}^{\text{nonres}}\frac{e^{i\Phi(\ve k)t}(k_3+k_4)}{k_1(k_1+k_2)(k_1+k_3+k_4)(k_2+k_3+k_4)}\phi_{k_1}\psi_{k_2}\xi_{k_3}\eta_{k_4}.
\end{equation*}
We are now able to write the cKdV~(\ref{eq:9}) in its \emph{second form}, namely
\begin{align}
  \label{dbp15}
  &\partial_t\left[
    \vect{u_k}{v_k}-\vect{B_2(u,v)_k-B_2(u,u)_k}{B_2(u,v)_k-B_2(v,v)_k} + \ve B_3(u,v)_k\right] \notag\\
  &=\ve R_{\text{3res}}(u,v)_k+\ve B_4(u,v)_k \:, \quad k\in\ZZ_0 \:,
\end{align}
where every term in $\ve B_3(u,v)_k$ has the format $\pm \frac{1}{36}B_3(\phi,\psi,\xi)_k$ with each of the three arguments being either $u$ or $v$. Hence, by the smoothing property of $B_3(\phi,\psi,\xi)$ provided in Lemma \ref{L:B3}, one has
\begin{align} \label{veB3}
\norm{\ve B_3(u,v)}_{(\dot H^{s+2})^2}\leq C(s) \norm{(u,v)}_{(\dot H^s)^2}^3, \text{\;\;for\;\;} s\geq 0.
\end{align}
On the other hand, each term in $\ve B_4(u,v)_k$ is either $\pm \frac{i}{72}B_4^1(\phi,\psi,\xi,\eta)$ or $\pm \frac{i}{72} B_4^2(\phi,\psi,\xi,\eta)$ with each of the four arguments being $u$ or $v$.
Due to Lemma \ref{L:B4}, the multi-linear operator $B_4(\phi,\psi,\xi,\eta)$ defined by
\begin{equation*}
  B_4(\phi,\psi,\xi,\eta) := B_4^1(\phi,\psi,\xi,\eta) + B_4^2(\phi,\psi,\xi,\eta)
\end{equation*}
has nice smoothing property, which yields
\begin{align} \label{veB4}
\norm{\ve B_4(u,v)}_{(\dot H^{s+\e})^2}\leq C(s,\e)\norm{(u,v)}_{(\dot H^s)^2}^4,
\end{align}
for $s\geq 0$ and $\e \in (0,\frac{1}{2})$.

\smallskip

\section{Global existence for $s>0$}\label{S-4}

	In this section we address the global existence of
	solutions of the cKdV (\ref{eq:9}). For this purpose we utilize a Galerkin
	version of equation~(\ref{eq:9}) which reads
\begin{align}
  \label{gal1}
  \begin{cases}
  \partial_t\vect{u_k^N}{v_k^N}=
  \frac{1}{2}i\mathcal P k \sum_{k_1+k_2=k}e^{3ikk_1k_2t}
  \vect{(\mathcal Pu_{k_1}^N)(\mathcal Pv_{k_2}^N)-(\mathcal Pu_{k_1}^N)(\mathcal Pu_{k_2}^N)}
  {(\mathcal Pu_{k_1}^N)(\mathcal Pv_{k_2}^N)-(\mathcal Pv_{k_1}^N)(\mathcal Pv_{k_2}^N)} \\
  u^N_k(0)=u^{\text{in}}_k \;,\;\;v^N_k(0)=v^{\text{in}}_k,
  \end{cases}
\end{align}
for $k\in \ZZ_0$.
Here $\mathcal P$ denotes the projection on the low
Fourier modes $|k|\leq  N$, that is,
\begin{align} \label{projection}
\mathcal Pu=\sum_{|k|\leq N}u_k e^{ikx} \text{\;\;and\;\;}
\mathcal Pu_k=(\mathcal P u)_k=\Big\{
\begin{array}{lcl} u_k&\text{if}& |k|\leq N  \\ 0&\text{if} & |k|>N \end{array}.
\end{align}
We stress that the operator $\mathcal P$ depends on $N$. For the sake of conciseness, we choose the notation $\mathcal P$ instead of $\mathcal P_N$.

It is easy to see from~(\ref{gal1}) that $\partial_t(u_k^N,v_k^N)=0$ for $|k|>N$.
Therefore~(\ref{gal1}) is effectively a finite system of ODEs.

The following proposition follows by standard arguments of ODE theory
since the nonlinearity in cKdV is locally Lipschitz.
\begin{proposition}\label{gal2}
Let the initial data $(u^{\emph{in}},v^{\emph{in}})\in (\dot
H^{0})^2$. For every positive integer $N$, there exists $T>0$
such that problem~(\ref{gal1}) has a unique solution $(u^N(t),v^N(t))\in (\dot H^{0})^2$
on the time interval $[0,T]$. The solution can be extended to a
maximal interval of existence $[0,T_{\max})$ such that either
$T_{\max}=+\infty$, or if $T_{\max}$ is finite then one has
$\limsup_{t\rightarrow T_{\max}^{-}}\norm{(u^N(t),v^N(t))}_{(\dot
H^0)^2}=+\infty$.
\end{proposition}

The next result shows that the solution $(u^N,v^N)$ of the Galerkin
system can not blow up in finite time in $(\dot H^0)^2$, and
therefore $T_{\max}=+\infty$.

\begin{proposition}\label{gal3}
	Let $(u^{\emph{in}},v^{\emph{in}})\in (\dot H^0)^2$. Then the solution
	$(u^N,v^N)$ of the Galerkin system (\ref{gal1}) exists globally in
	time. Furthermore, the quantity
\begin{align*}
\mathcal E(u^N(t),v^N(t)):= 2\norm{u^N(t)}_{\dot H^0}^2+2\norm{v^N(t)}_{\dot H^0}^2+\norm{u^N(t)-v^N(t)}_{\dot H^0}^2
\end{align*}
is conserved in time.
\end{proposition}

\begin{proof}
By Proposition \ref{gal2}, it is sufficient to show the conservation of $\mathcal E(u^N(t),v^N(t))$ on the interval $[0,T_{\max})$. Thus, in the proof, we consider $t\in [0,T_{\max})$.
 Let us rewrite the quantity $\mathcal E(u^N(t),v^N(t))$ as
\begin{equation*}
\mathcal E(u^N(t),v^N(t)) = 3\norm{u^N(t)}_{\dot H^0}^2+3\norm{v^N(t)}_{\dot H^0}^2-2\SP{u^N(t)}{v^N(t)} \:,
\end{equation*}
    where $\SP{\cdot}{\cdot}$ denotes the scalar
	product in $\dot H^0$.
	Differentiating $\mathcal E(u^N(t),v^N(t))$ yields
\begin{align*}
    \partial_t \mathcal E(u^N(t),v^N(t)) & = 6\SP{\partial_t u^N}{u^N}
	+ 6\SP{\partial_t v^N}{v^N} -2\SP{\partial_t u^N}{v^N}
	-2\SP{\partial_t v^N}{u^N} \notag\\
    & = \sum_{0<|k|\leq  N}
    6\partial_t u_k^N {u}^N_{-k} + 2\partial_t v_k^N {v}_{-k}^N - 2\partial_t u_k^N {v}^N_{-k} - 2\partial_t v_k^N {u}^N_{-k} \:,
\end{align*}
where we have used $\bar u^N_k=u^N_{-k}$ and the fact $\partial_t(u_k^N,v_k^N)=0$ for $|k|>N$.

	Plugging in the expressions of the
	derivatives of $u_k^N$ and $v_k^N$ by using the Galerkin
	system~(\ref{gal1}), we obtain after collecting similar terms
\begin{align} \label{gal8}
    \partial_t \mathcal E(u^N(t),v^N(t)) & = \sum_{k_1+k_2+k_3=0 \atop 0<|k_1|,|k_2|,|k_3|\leq N}
    i k_3 e^{-3ik_1k_2k_3t}\big\{
    3 u_{k_1}^N u_{k_2}^N u_{k_3}^N+3 v_{k_1}^N v_{k_2}^N v_{k_3}^N \notag\\
    & \qquad  -2 u_{k_1}^N v_{k_2}^N v_{k_3}^N-2 u_{k_1}^N v_{k_2}^N u_{k_3}^N  - u_{k_1}^N u_{k_2}^N v_{k_3}^N- v_{k_1}^N v_{k_2}^N u_{k_3}^N \big\} \notag\\
    & =: A \:.
\end{align}
	Now, writing $k_3=-(k_1+k_2)$ one has
\begin{equation*}
  \begin{split}
    \partial_t \mathcal E(u^N(t),v^N(t))& = -\sum_{k_1+k_2+k_3=0 \atop 0<|k_1|,|k_2|,|k_3|\leq N}
    i k_1 e^{-3ik_1k_2k_3t}\big\{
    3 u_{k_1}^N u_{k_2}^N u_{k_3}^N +3 v_{k_1}^N v_{k_2}^N v_{k_3}^N \\
    & \qquad \qquad-2 u_{k_1}^N v_{k_2}^N v_{k_3}^N -2 u_{k_1}^N v_{k_2}^N u_{k_3}^N  - u_{k_1}^N u_{k_2}^N v_{k_3}^N- v_{k_1}^N v_{k_2}^N u_{k_3}^N \big\} \\
    & \hspace{0.2 in} -\sum_{k_1+k_2+k_3=0 \atop 0<|k_1|,|k_2|,|k_3|\leq N}
    i k_2 e^{-3ik_1k_2k_3t}\big\{
    3 u_{k_1}^N u_{k_2}^N u_{k_3}^N +3 v_{k_1}^N v_{k_2}^N v_{k_3}^N \\
    & \qquad \qquad -2 u_{k_1}^N v_{k_2}^N v_{k_3}^N-2 u_{k_1}^N v_{k_2}^N u_{k_3}^N  - u_{k_1}^N u_{k_2}^N v_{k_3}^N-v_{k_1}^N v_{k_2}^N u_{k_3}^N \big\} .
  \end{split}
\end{equation*}
	An exchange of $k_1$ and $k_3$ in the first sum and of $k_2$
	and $k_3$ in the second sum yields, after collecting similar terms
\begin{align}
  \label{gal8b}
    \partial_t \mathcal E(u^N(t),v^N(t)) & = -\sum_{k_1+k_2+k_3=0 \atop 0<|k_1|,|k_2|,|k_3|\leq N}
    i k_3  e^{-3ik_1k_2k_3t} \big\{
    6 u_{k_1}^N u_{k_2}^N u_{k_3}^N+6 v_{k_1}^N v_{k_2}^N v_{k_3}^N  \notag\\
    & \hspace{0.3 in} -4 u_{k_1}^N v_{k_2}^N v_{k_3}^N- 4 u_{k_1}^N v_{k_2}^N u_{k_3}^N -2 u_{k_1}^N u_{k_2}^N v_{k_3}^N -2 v_{k_1}^N v_{k_2}^N u_{k_3}^N \big\} \notag\\
    & = -2A \:.
\end{align}

Comparing (\ref{gal8}) and (\ref{gal8b}) yields $A=-2A$, and thus $A=0$, i.e., the quantity $\mathcal E(u^N(t),v^N(t))$ is conserved.

\end{proof}

	Now we address estimates for higher order Sobolev norms of the global solution of the Galerkin system (\ref{gal1}). In
	order to do that we utilize the second form of the
	equation. Taking the solution $(u^N,v^N)$  of our Galerkin
	system~(\ref{gal1}) we see that it satisfies the Galerkin
	version of the second form (\ref{dbp15}) of the cKdV
	introduced in Section~\ref{sec-dbp}, namely
\begin{align}
  \label{gal10}
  &\partial_t\left[
    \vect{u^N_k}{v^N_k}-\vect{B_2^N(u^N,v^N)_k-B_2^N(u^N,u^N)_k}{B_2^N(u^N,v^N)_k-B_2^N(v^N,v^N)_k} + \ve B_3^N(u^N,v^N)_k\right] \notag\\
  &=\ve R_{\text{3res}}^N(u^N,v^N)_k+\ve B_4^N(u^N,v^N)_k \:, \quad k\in\ZZ_0 \:,
\end{align}
where we use the notations
\begin{align*}
  \label{gal11}
  &B_2^N(\phi,\psi):=\mathcal P B_2(\mathcal P \phi,\mathcal P \psi) \:, \quad   \ve B_3^N(\phi,\psi):=\mathcal P \ve B_3 (\mathcal P \phi,\mathcal P \psi) \:,\notag\\
   &\ve R_{\text{3res}}^N(\phi,\psi):=\mathcal P \ve R_{\text{3res}} (\mathcal P \phi,\mathcal P \psi) \:, \quad
   \ve B_4^N(\phi,\psi):=\mathcal P \ve B_4 (\mathcal P \phi,\mathcal P \psi) \:.
\end{align*}

The next result states, for any $T>0$, the high order Sobolev norms of the Galerkin system solutions $(u^N,v^N)$ are bounded on $[0,T]$ uniformly in $N$.
\begin{proposition} \label{high-sobolev}
	Assume $s\geq  0$, $(u^{\emph{in}},v^{\emph{in}})\in (\dot
	H^{s})^2$, and $T>0$. Let $(u^N(t),v^N(t))$ be the solution of the Galerkin
	system~(\ref{gal1}) over the interval $[0,T]$ with the initial data $(u^{\emph{in}},v^{\emph{in}})$. Then
	$(u^N(t),v^N(t))$ solves~(\ref{gal10}) and satisfies the
	estimate
\begin{equation}
    \label{gal14}
    \norm{(u^N(t),v^N(t))}_{(\dot H^{s})^2}\leq  C\left(\norm{(u^{\emph{in}},v^{\emph{in}})}_{(\dot H^0)^2},T,s\right) \:, \quad \text{for all} \quad t\in [0,T]\:,
\end{equation}
where the bound is independent of $N$.
\end{proposition}

\begin{proof}
	Due to the fact that we are dealing with a finite
	number of ODEs (and finite sums) we observe by straight
	forward calculation (differentiation by parts twice)  as in Section \ref{first-dbp} and Section~\ref{sec-dbp} that a
	solution of ~(\ref{gal1}) also solves~(\ref{gal10}).

	Throughout, we consider $t\in [0,T]$.
By (\ref{gal1}) (or (\ref{gal10})), it is clear that, $\partial_t (u_k^N,v_k^N)=0$ for $|k|>N$. Therefore,
\begin{equation}
  \label{gal15}
  \norm{(I-\mathcal P)(u^N(t),v^N(t))}_{(\dot H^{s})^2} = \norm{(I-\mathcal P)(u^{\text{in}},v^{\text{in}})}_{(\dot H^{s})^2}\leq
  \norm{(u^{\text{in}},v^{\text{in}})}_{(\dot H^{s})^2} \:.
\end{equation}
It follows that $\norm{(u^N(t),v^N(t))}_{(\dot H^{s})^2}$ is bounded for every $t\in [0,T]$. Our goal, however, is to show that the bound on the $H^s$-norm is uniform in $N$.

In~(\ref{gal10}) we set
\begin{equation}
  \label{gal16}
  \ve z^N_k:= \vect{u^N_k}{v^N_k}-\vect{B_2^N(u^N,v^N)_k-B_2^N(u^N,u^N)_k}{B_2^N(u^N,v^N)_k-B_2^N(v^N,v^N)_k} + \ve B_3^N(u^N,v^N)_k,
\end{equation}
then (\ref{gal10}) goes over to
\begin{align}
  \label{gal20}
  \partial_t \ve z^N_k=\ve R_{\text{3res}}^N(u^N,v^N)_k+\ve B_4^N(u^N,v^N)_k,\;\;k\in \ZZ_0\:.
\end{align}

	Since $\ve B_4$ gives a maximal
	gain of $\eps<1/2$ spatial derivatives according to (\ref{veB4}), we fix a positive integer $n_0$ such that
	$s/n_0=\eps<1/2$. Once we establish the uniform bound in $(\dot H^{\eps})^2$, we can iterate the argument and after $n_0$ steps we will have the desired bound
	in $(\dot H^{s})^2$.

Put $M_0:=\norm{(u^{\text{in}},v^{\text{in}})}_{(\dot H^{0})^2}$. Due to the conservation law established in Proposition \ref{gal3}, one has
\begin{align}  \label{H0bound}
\norm{(u^N(t),v^N(t))}_{(\dot H^{0})^2}\leq  C (M_0), \text{\;\;for all\;\;}  t\in [0,T].
\end{align}
 Here $C(M_0)$ is a constant depending on $M_0$, and it may change hereafter from line to line. It is clear that the mapping properties of $\ve R_{\text{3res}}^N$ and $\ve B_4^N$ are the same as the ones of $\ve R_{\text{3res}}$ and $\ve B_4$, and thus by (\ref{resbound}), (\ref{veB4}) and (\ref{H0bound}), we obtain
\begin{equation}
  \label{gal21}
  \norm{\ve R_{\text{3res}}^N(u^N,v^N)}_{(\dot H^{\eps})^2}
  +\norm{\ve B_4^N(u^N,v^N)}_{(\dot H^{\eps})^2}
  \leq  C(M_0)\big(\norm{(u^N,v^N)}_{(\dot H^{\eps})^2}+1\big).
\end{equation}

Taking into account equation~(\ref{gal16}) and the smoothing
properties of $B_2$ and $\ve B_3$ (see Lemma~\ref{L:B2} and the estimate (\ref{veB3})),
we have
\begin{equation}
  \label{gal22}
  \norm{(u^N,v^N)}_{(\dot H^{\eps})^2} \leq
   \norm{\ve z^N}_{(\dot H^{\eps})^2}+C(M_0),
\end{equation}
	and therefore together with~(\ref{gal21}), one has
\begin{equation}
  \label{gal23}
  \norm{\ve R_{\text{3res}}^N(u^N,v^N)}_{(\dot H^{\eps})^2}
  +\norm{\ve B_4^N(u^N,v^N)}_{(\dot H^{\eps})^2} \leq C(M_0)(\norm{\ve z^N}_{(\dot H^{\eps})^2}+1).
\end{equation}
Then, we see from (\ref{gal20}) and (\ref{gal23}) that
\begin{align}  \label{partialz}
\norm{\partial_t \ve z^N}_{(\dot H^{\eps})^2}
\leq C(M_0)(\norm{\ve z^N}_{(\dot H^{\eps})^2}+1).
\end{align}

By (\ref{gal16}) and the fact that $\norm{(u^N(t),v^N(t))}_{(\dot H^{s})^2}<\infty$ for each $t\in [0,T]$, it is clear that $\norm{\ve z^N(t)}_{(\dot H^{s})^2}$ is also finite for every $t\in [0,T]$. Thus, we calculate
\begin{align*}
  \label{gal24}
  \partial_t \norm{\ve z^N}_{(\dot H^{\eps})^2}^2
  = 2\left\langle \partial_t \ve z^N, \ve z^N   \right\rangle_{(\dot H^{\eps})^2}
  &\leq \norm{\partial_t \ve z^N}_{(\dot H^{\eps})^2}^2+\norm{\ve z^N}_{(\dot H^{\eps})^2}^2 \notag\\
  & \leq C(M_0)(\norm{\ve z^N}_{(\dot H^{\eps})^2}^2+1),
\end{align*}
where (\ref{partialz}) has been used.
Then, by means of Gronwall's inequality, we have
\begin{equation*}
  \label{gal28}
  \norm{\ve z^N(t)}_{(\dot H^{\eps})^2}
  \leq  C(M_0,T) \:, \quad t\in[0,T] ,
\end{equation*}
and along with (\ref{gal22}), it follows that
\begin{align*}
 \norm{(u^N(t),v^N(t))}_{(\dot H^{\eps})^2}  \leq C(M_0,T)\:, \quad t\in[0,T] .
\end{align*}
Finally, we iterate the above argument $n_0$ times and conclude
\begin{align*}
 \norm{(u^N(t),v^N(t))}_{(\dot H^{s})^2}  \leq C(M_0,T,s)\:, \quad t\in[0,T] ,
\end{align*}
where the bound is uniform in $N$. It is worth to mention that, the uniform bound above depends on the $\dot H^0$-norm, and not the $\dot H^s$-norm of the initial data.
\end{proof}

We now establish the existence of global solutions (without uniqueness) which is stated in Theorem \ref{exist} for the case
	$s>0$. Also we show the conservation law (\ref{conserve}) and the bound (\ref{gal31}).
The uniqueness and continuous dependence on initial data for $s>0$ will be justified in Section \ref{regular} and \ref{irregular}. The case $s=0$ will be treated in Section \ref{irregular}.

\begin{proof}
Let $s>0$ be given. As before, we let $(u^N(t),v^N(t))$ be the solution of the Galerkin system (\ref{gal1}) on $[0,T]$. Taking some
	$\theta>3/2$, thanks to (\ref{gal1}) and the mapping property of $B_1$ provided in Lemma \ref{L:B1} as well as the conservation law in Proposition \ref{gal3}, we obtain
  \begin{equation*}
    \norm{\partial_t(u^N(t),v^N(t))}_{(\dot H^{-\theta})^2}\leq
    C(\theta)\norm{(u^N(t),v^N(t))}_{(\dot H^{0})^2}^2 \leq  \tilde C(\theta)\norm{(u^{\text{in}},v^{\text{in}})}_{(\dot H^0)^2}^2 \:,
\end{equation*}
that is, $\partial_t(u^N(t),v^N(t))$ is
	bounded in $L^\infty([0,T];(\dot H^{-\theta})^2) \subset
	L^p([0,T];(\dot H^{-\theta})^2)$ uniformly with respect to
	$N$. Furthermore, by virtue of Proposition \ref{high-sobolev}
	we observe that the sequence $(u^N,v^N)$ is uniformly bounded in
	$L^\infty([0,T];(\dot H^{s})^2)\subset L^p([0,T];(\dot
	H^{s})^2)$. Therefore, due to Aubin's Compactness Theorem, for $0<s_0<s$,
    there exists a subsequence, which is still denoted by $(u^N, v^N)$, converging strongly to $(u,v)$ in $L^p([0,T];(\dot H^{s_0})^2)$ and
	$\ast$-weakly in $L^\infty([0,T];(\dot H^{s})^2)$, and along with (\ref{gal14}), we infer
(\ref{gal31}) holds.

	Since the subsequence $(u^N,v^N)$ converges strongly to $(u,v)$ in $L^p([0,T];(\dot
	H^0)^2)$, we can extract a further subsequence $(u^N(t),v^N(t))$
	converging strongly to $(u(t),v(t))$ in $(\dot H^{0})^2$ for almost every $t\in
	[0,T]$. Moreover, by means of Proposition \ref{gal3},
we have $\mathcal E(u^N(t),v^N(t))=\mathcal E(u^{\text{in}},v^{\text{in}})$
	for every $t$, and therefore
	$\mathcal E(u(t),v(t))=\mathcal E(u^{\text{in}},v^{\text{in}})$ for almost every
	$t\in [0,T]$.

	Now we must show that the weak limit $(u,v)$ is
	indeed a solution of the cKdV~(\ref{eq:9}) in the sense of
	Definition~\ref{dfn-solution}. For this purpose we utilize the
	fact that each $(u_k^N,v_k^N)$ is a solution of the Galerkin system (\ref{gal1})
	and hence a solution of
\begin{align}
  \label{gal33}
  &\vect{u_k^N(t)}{v_k^N(t)}-\vect{u^{\text{in}}_k}{v^{\text{in}}_k} \notag\\
  &=\int_0^t
  \vect{\mathcal PB_1(\mathcal Pu^N(\tau),\mathcal Pv^N(\tau))_k-\mathcal PB_1(\mathcal Pu^N(\tau),\mathcal Pu^N(\tau))_k}
  {\mathcal PB_1(\mathcal Pu^N(\tau),\mathcal Pv^N(\tau))_k-\mathcal PB_1(\mathcal Pv^N(\tau),\mathcal Pv^N(\tau))_k} \d\tau \:.
\end{align}
	Using the symmetry of $B_1$ and setting $\mathcal Q:=I-\mathcal P$ we
	can rewrite~(\ref{gal33}) as
\begin{align}
  \label{gal34}
    &\vect{u_k^N(t)}{v_k^N(t)}-\vect{u^{\text{in}}_k}{v^{\text{in}}_k} \notag\\
    &=\int_0^t \vect{\mathcal PB_1(\mathcal Pu^N,\mathcal P(v^N-v))_k+\mathcal PB_1(\mathcal P(u^N-u),\mathcal Pv)_k}
    {\mathcal PB_1(\mathcal Pu^N,\mathcal P(v^N-v))_k+\mathcal PB_1(\mathcal P(u^N-u),\mathcal Pv)_k} \d\tau  \notag\\
    & \hspace{0.2 in} -\int_0^t \vect{\mathcal PB_1(\mathcal P(u^N-u),\mathcal P(u^N+u))_k+\mathcal PB_1(u,\mathcal Qv)_k+\mathcal PB_1(\mathcal Qu,\mathcal Pv)_k}
    {\mathcal PB_1(\mathcal P(v^N-v),\mathcal P(v^N+v))_k+\mathcal PB_1(u,\mathcal Qv)_k+\mathcal PB_1(\mathcal Qu,\mathcal Pv)_k}\d\tau \notag\\
    &\hspace{0.2 in} +\int_0^t \vect{\mathcal PB_1(\mathcal Qu,\mathcal Pu+u)_k-\mathcal QB_1(u,v)_k+\mathcal QB_1(u,u)_k}
    {\mathcal PB_1(\mathcal Qv,\mathcal Pv+v)_k-\mathcal QB_1(u,v)_k+\mathcal QB_1(v,v)_k}\d\tau \notag\\
    &\hspace{0.2 in} +\int_0^t \vect{B_1(u,v)_k-B_1(u,u)_k}{B_1(u,v)_k-B_1(v,v)_k} \d\tau \:,\;\;\;\; k\in \ZZ_0.
\end{align}
	First we observe that due to the
	convergence of the subsequence $(u^N,v^N)\rightarrow (u,v)$ strongly in the
	space $L^p([0,T];(\dot H^{s_0})^2)$ and due to the fact
	$(u,v)\in L^\infty([0,T];(\dot H^{s})^2)$, we
	see that the first three integral terms on the right-hand side
	are finite and converge to zero as
	$N\rightarrow\infty$. We demonstrate this for a typical term $\int_0^t \mathcal PB_1(\mathcal Pu^N,\mathcal P(v^N-v))_k\d\tau$ (terms
	of the same structure are of course treated similarly). By using Cauchy-Schwarz inequality, for $\theta>3/2$, we deduce
\begin{equation}
  \label{gal35}
  \begin{split}
    &\left| \int_0^t\mathcal PB_1(\mathcal Pu^N,\mathcal P(v^N-v))_k  \d\tau \right|  \\
    &\leq   |k|^{\theta}
    \int_0^t\left(\sum_{j\in\ZZ_0} |j|^{-2\theta}|\mathcal PB_1(\mathcal Pu^N,\mathcal P(v^N-v))_j|^2\right)^{1/2}\d\tau  \\
    &\leq  |k|^{\theta}\int_0^t \norm{B_1(\mathcal Pu^N,\mathcal P(v^N-v))}_{\dot H^{-\theta}} \d\tau  \\
    &\leq  C|k|^{\theta}\int_0^t \norm{u^N}_{\dot H^0}\norm{v^N-v}_{\dot H^0} \d\tau  \\
    &\leq  \tilde C|k|^{\theta}\norm{v^N-v}_{L^1([0,T];\dot H^0)} \longrightarrow 0 \quad \text{as $N\rightarrow\infty$ \:,}
  \end{split}
\end{equation}
	where we have used the mapping property of $B_1$ provided in Lemma~\ref{L:B1}, and the uniform boundedness of the $\dot H^0$-norm of $u^N$.
Next we treat (similar terms are again treated
	similarly) the term $\mathcal QB_1(u,v)_k$.
\begin{equation*}
  \label{gal36}
  \begin{split}
    \norm{\mathcal QB_1(u,v)}_{\dot H^{-\theta}}^2 &
    \leq  C\sum_{|k|>N}|k|^{2(1-\theta)}\left(\sum_{k=k_1+k_2}|u_{k_1}||v_{k_2}|\right)^2 \\
    & \leq  C\norm{u}_{\dot H^0}^2\norm{v}_{\dot H^0}^2\sum_{|k|>N}|k|^{2(1-\theta)}  \longrightarrow 0 \quad \text{as $N\rightarrow\infty$} \:,
  \end{split}
\end{equation*}
	due to $\theta>3/2$. Thus, with the similar arguments as in~(\ref{gal35}) we derive
\begin{equation}
  \label{gal37}
  \left| \int_0^t\mathcal QB_1(u,v)_k\d\tau \right|
  \leq  |k|^{\theta} \int_0^t   \norm{\mathcal QB_1(u,v)}_{\dot H^{-\theta}} \d\tau
  \longrightarrow 0, \;\;\text{as}\;\;N\rightarrow\infty,
\end{equation}
	where we have used Lebesgue's Dominated
	Convergence Theorem with $\norm{u(t)}_{\dot
	H^0}\norm{v(t)}_{\dot H^0}$ (which is bounded a.e.) as
	a majorant. The remaining terms are treated in the same
	manner. Passing to the limit in~(\ref{gal34}) and using that the subsequence
	$(u_k^N(t),v_k^N(t))\rightarrow (u_k(t),v_k(t))$
	for each fixed $k\in\ZZ_0$ a.e. on $[0,T]$ we
	obtain
\begin{equation}
  \label{gal38}
  \vect{u_k(t)}{v_k(t)}-\vect{u^{\text{in}}_k}{v^{\text{in}}_k}=
  \int_0^t
  \vect{B_1(u(\tau),v(\tau))_k-B_1(u(\tau),u(\tau))_k}
  {B_1(u(\tau),v(\tau))_k-B_1(v(\tau),v(\tau))_k} \d\tau \:,
\end{equation}
	which is true for almost every $t\in [0,T]$. Since
	$u_k(t)$ and $v_k(t)$ is absolutely continuous over the interval $[0,T]$
	(see Remark~\ref{rmk-abscont}), this identity holds for every
	$t\in [0,T]$. That is $(u,v)$ is indeed a
	solution of the cKdV (\ref{eq:9}) in the sense of Definition~\ref{dfn-solution}.
\end{proof}

\smallskip

\section{Uniqueness for $s>1/2$} \label{regular}
In the previous section we established global existence,
without uniqueness, of solutions to the cKdV
system~(\ref{eq:9}) in the space $(\dot H^s)^2$ for $s>0$. Here, we will using the Banach Fixed Point Theorem to establish the uniqueness of solutions, as well as the continuous dependence on initial data for $s>1/2$. The case $s\in [0,1/2]$ will be treated in the next section.

In~\cite{B-I-T}, where the periodic KdV was studied, the authors also used the contraction mapping argument to establish the uniqueness. However, their technique depends on the invertibility of a linear operator, which relies on the fact
that one can solve a 1d-boundary value problem for an ODE explicitly and estimate its solution. But such method is infeasible to adopt here for our cKdV system, since the linearization of the left-hand side of (\ref{eq:13}) may not be invertible,
where the difficulty lies in explicitly solving a
	boundary value problem for a system of two coupled ODEs in
	which the situation is much more complicated.
In order to bypass this obstacle, we split solutions properly into high and low Fourier modes, and recast the differentiation by parts procedure to terms involving high frequencies for the sake of taking advantage of the time-averaging induced squeezing. Similar idea was also used in \cite{B-I-T} to treat the so-called less regular initial data ($0\leq s\leq 1/2$), and in \cite{KwoOh} (following \cite{B-I-T})
to study the unconditional uniqueness of the modified KdV equation for $s\geq 1/2$. We believe that this kind of approach is more natural and general, especially for systems, which avoids studying the invertibility of a linear operator, so it is easier to implement to other dispersive equations for establishing uniqueness of solutions. In fact, one of the main purposes of this paper is to demonstrate this idea for such a typical system.

Let $N\geq 1$ be an integer that will be selected later. Recall, $\mathcal P$ defined in (\ref{projection}) denotes the projection on the low Fourier modes $|k|\leq N$. In addition,
we define $\mathcal Q=I-\mathcal P$, where $I$ is the identity map. Observe that $\mathcal P$ and $\mathcal Q$ both depend on $N$.

We decompose $B_1(u,v)$ by splitting the Fourier modes of $u$ and $v$ into high and low modes. More precisely,
\begin{align*}
B_1(u,v)=B_1(\mathcal Pu,\mathcal Pv)+\left[B_1(\mathcal Pu,\mathcal Qv)+B_1(\mathcal Qu,v)\right].
\end{align*}
Thus, the original cKdV (\ref{eq:10}) can be written as
\begin{align} \label{kdv-1}
\partial_t \vect{u}{v}=\ve B_1^P(u,v)+\ve B_1^Q(u,v)
\end{align}
where vector functions $\ve B_1^P(u,v)$ and $\ve B_1^Q(u,v)$ are defined by
\begin{align} \label{B1-P}
\ve B_1^P(u,v):=\vect{B_1(\mathcal Pu,\mathcal Pv)-B_1(\mathcal Pu,\mathcal Pu)}{B_1(\mathcal Pu,\mathcal Pv)-B_1(\mathcal Pv,\mathcal Pv)},
\end{align}
and
\begin{align} \label{B1-Q}
\ve B_1^Q(u,v):=\vect{B_1(\mathcal Pu,\mathcal Qv)+B_1(\mathcal Qu,v)}{B_1(\mathcal Pu,\mathcal Qv)+B_1(\mathcal Qu,v)}
-\vect{B_1(\mathcal Pu,\mathcal Qu)+B_1(\mathcal Qu, u)}{B_1(\mathcal Pv,\mathcal Qv)+B_1(\mathcal Qv, v)}.
\end{align}

Unlike the first differentiation by parts performed in Section \ref{first-dbp}, now we apply the differentiation by parts procedure for $\ve B_1^Q(u,v)$ only, and leave $\ve B_1^P(u,v)$ untouched. We demonstrate the computation for a typical term $B_1(\mathcal Pu,\mathcal Qv)$.
In fact, for $k\in \ZZ_0$,
\begin{align} \label{u-1}
B_1(\mathcal Pu,\mathcal Qv)_k&=\frac{1}{2}ik\sum_{k_1+k_2=k}e^{3ikk_1k_2t}\mathcal Pu_{k_1} \mathcal Qv_{k_2} \notag\\
&=\frac{1}{6}\partial_t \left(\sum_{k_1+k_2=k}\frac{e^{3ikk_1k_2t}\mathcal Pu_{k_1}\mathcal Qv_{k_2}}{k_1 k_2}\right) \notag\\
&\hspace{0.2 in}-\frac{1}{6}\sum_{k_1+k_2=k}\frac{e^{3ikk_1k_2t}}{k_1 k_2}(\mathcal Pu_{k_1}\partial_t \mathcal Qv_{k_2}+\mathcal Qv_{k_2}\partial_t \mathcal Pu_{k_1}).
\end{align}

If we denote
\begin{align*}
\mathcal P(u_{k_1} v_{k_2})=\Big\{
\begin{array}{lcl} u_{k_1} v_{k_2}&\text{if}& |k_1+k_2|\leq N  \\ 0&\text{if} & |k_1+k_2|>N \end{array}
\end{align*}
and
\begin{align*}
\mathcal Q(u_{k_1} v_{k_2})=\Big\{
\begin{array}{lcl} u_{k_1} v_{k_2}&\text{if}& |k_1+k_2|> N  \\ 0&\text{if} & |k_1+k_2|\leq N, \end{array}
\end{align*}
then by (\ref{eq:9}),
\begin{align*}
\partial_t \mathcal Pu_k=\frac{1}{2}ik\sum_{k_1+k_2=k}e^{3ikk_1k_2t}(\mathcal P(u_{k_1}v_{k_2})-\mathcal P(u_{k_1}u_{k_2})),
\end{align*}
and
\begin{align*}
\partial_t \mathcal Qv_k = \frac{1}{2}ik\sum_{k_1+k_2=k}e^{3ikk_1k_2t}\left(\mathcal Q(u_{k_1}v_{k_2})-\mathcal Q(v_{k_1}v_{k_2})\right).
\end{align*}
Therefore, for $k\in \ZZ_0$,
\begin{align} \label{u-2}
&-\frac{1}{6}\sum_{k_1+k_2=k}\frac{e^{3ikk_1k_2t}}{k_1 k_2}(\mathcal Pu_{k_1}\partial_t \mathcal Qv_{k_2}+\mathcal Qv_{k_2}\partial_t \mathcal Pu_{k_1})\notag\\
&=-\frac{1}{12}i \sum_{k_1+k_2+k_3=k}\frac{e^{3i(k_1+k_2)(k_2+k_3)(k_3+k_1)t}}{k_1}
\big[\mathcal Pu_{k_1} \mathcal Q(u_{k_2}v_{k_3})
-\mathcal Pu_{k_1}  \mathcal Q(v_{k_2}v_{k_3}) \notag\\
&\hspace{2.5 in}+\mathcal Qv_{k_1} \mathcal P(u_{k_2}v_{k_3})
-\mathcal Qv_{k_1} \mathcal P(u_{k_2}u_{k_3})\big] \notag\\
&=:f_k.
\end{align}

Combining (\ref{u-1}) and (\ref{u-2}) yields
\begin{align*}
B_1(\mathcal Pu,\mathcal Qv)_k
=\partial_t B_2(\mathcal Pu,\mathcal Qv)_k +f_k, \;\;k\in \ZZ_0,
\end{align*}
where $f_k$ is defined in (\ref{u-2}). Similarly, we can apply differentiation by parts for all terms in  $\ve B_1^Q(u,v)$ defined in (\ref{B1-Q}), and obtain the \emph{modified first form of the cKdV}:
\begin{align} \label{m-1st-form}
\partial_t \left[\vect{u}{v}-\ve B_2^Q(u,v)\right] =\ve B_1^P(u,v)+ \ve R_3^Q(u,v),
\end{align}
where $\ve B_1^P(u,v)$ is defined in (\ref{B1-P}) and $\ve B_2^Q(u,v)$ is defined by
\begin{align} \label{B2-Q}
\ve B_2^Q(u,v):=\vect{B_2(\mathcal Pu,\mathcal Qv)+B_2(\mathcal Qu, v)}
{B_2(\mathcal Pu,\mathcal Qv)+B_2(\mathcal Qu, v)}-\vect{B_2(\mathcal Pu,\mathcal Qu)+B_2(\mathcal Qu, u)}{B_2(\mathcal Pv,\mathcal Qv)+B_2(\mathcal Qv, v)}.
\end{align}
For the sake of conciseness, we do not provide the exact formula of $\ve R_3^Q(u,v)$. But, notice that $f_k$ defined in (\ref{u-2}) is a typical part of $\ve R_3^Q(u,v)_k$. Thus, all terms in each components of $\ve R_3^Q(u,v)_k$ can be written in the form
\begin{align} \label{u-3}
\pm \frac{1}{12}i \sum_{k_1+k_2+k_3=k \atop \{k_1,k_2,k_3\} \in \mathcal D}\frac{e^{3i(k_1+k_2)(k_2+k_3)(k_3+k_1)t}}{k_1}
\phi_{k_1}\psi_{k_2}\xi_{k_3}, \;\;k\in \ZZ_0,
\end{align}
where $\mathcal D\subset \ZZ_0^3$ is a set of indices that might vary for different terms in $\ve R_3^Q(u,v)$,
and each of $\phi$, $\psi$, $\xi$ is either $u$ or $v$.
For instance, considering the first term of $f_k$ defined in (\ref{u-2}),
then $(\phi,\psi,\xi)=(u,u,v)$, and
the summation is carried out over $\mathcal D=\{\{k_1,k_2,k_3\} \in \ZZ_0^3: |k_1|\leq N, \; |k_2+k_3|> N\}$.

Since (\ref{u-3}) is essentially $R_3(\phi,\psi,\xi)$ with summation over a set $\mathcal D$, by the mapping property of $R_3$ provided in Lemma \ref{L:R3} we have
\begin{align} \label{est-til-R3}
\norm{\ve R_3^Q(u,v)}_{(\dot H^s)^2}
\leq C(s) \norm{(u,v)}_{(\dot H^s)^2}^3,
\end{align}
and
\begin{align} \label{est-til-R3'}
&\norm{\ve R_3^Q(u,v)-\ve R_3^Q(\tilde u,\tilde v)}_{(\dot H^s)^2} \notag\\
&\leq C(s) \norm{(u,v)-(\tilde u,\tilde v)}_{(\dot H^s)^2}
\left(\norm{(u,v)}^2_{(\dot H^s)^2}+ \norm{(\tilde u,\tilde v)}^2_{(\dot H^s)^2}\right),
\end{align}
for $s>1/2$.

Concerning $\ve B_1^P(u,v)$, the following result shows that the smoothing property of $\ve B_1^P$ is better than the one of $B_1$ provided in Lemma \ref{L:B1}.
\begin{lemma}
For $s\geq 0$, the operator $\ve B_1^P$ defined in (\ref{B1-P}) maps $\dot H^0 \times \dot H^0$ into $\dot H^s \times \dot H^s$ and satisfy
\begin{align} \label{est-low-B1'}
\norm{\ve B_1^P(u,v)}_{(\dot H^s)^2}\leq C(s,N) \norm{(u,v)}_{(\dot H^0)^2}^2,
\end{align}
and
\begin{align} \label{est-low-B1''}
&\norm{\ve B_1^P(u,v)-\ve B_1^P(\tilde u,\tilde v)}_{(\dot H^s)^2} \notag\\
&\leq C(s,N) \norm{(u,v)-(\tilde u,\tilde v)}_{(\dot H^0)^2}\left(\norm{(u,v)}_{(\dot H^0)^2}+\norm{(\tilde u,\tilde v)}_{(\dot H^0)^2}\right).
\end{align}
\end{lemma}

\begin{proof}
We consider a typical term $B_1(\mathcal Pu,\mathcal Pv)$. The estimates of the rest terms are similar.
Indeed, by the definition (\ref{eq:10a}) of $B_1$,
\begin{align*}
\norm{B_1(\mathcal Pu,\mathcal Pv)}^2_{\dot H^s}
&\leq \frac{1}{4} \sum_{|k|\leq 2N}|k|^{2+2s}\left(\sum_{\substack{k_1+k_2=k \atop |k_1|, |k_2|\leq N}}|u_{k_1}||v_{k_2}| \right)^2 \notag\\
& \leq \frac{1}{4}(2N)^{2+2s}(2N+1)\norm{u}_{\dot H^0}^2 \norm{v}_{\dot H^0}^2.
\end{align*}
In addition, since $B_1(\phi,\psi)$ is a bilinear operator, it follows that
\begin{align*}
&\norm{B_1(\mathcal Pu,\mathcal Pv)-B_1(\mathcal P\tilde u,\mathcal P\tilde v)}_{\dot H^s} \notag\\
&\leq C(s,N)(\norm{u-\tilde u}_{\dot H^0} \norm{v}_{\dot H^0}+\norm{v-\tilde v}_{\dot H^0} \norm{\tilde u}_{\dot H^0}).
 \end{align*}
\end{proof}

Furthermore, the operator $\ve B_2^Q$ defined in (\ref{B2-Q}) has the following mapping property stated in Lemma \ref{lem-1}, which indicates that the corresponding constant decreases to zero as $N\rightarrow \infty$.
This reflects the time-averaging induced squeezing.
\begin{lemma} \label{lem-1}
For any real number $s\geq 0$, the operator $\ve B_2^Q$ defined in (\ref{B2-Q}) maps $(\dot H^s)^2$ into $(\dot H^s)^2$ and satisfies
\begin{align} \label{est-til-B2}
\norm{\ve B_2^Q(u,v)}_{(\dot H^s)^2}\leq C(s)\frac{1}{N}\norm{(u,v)}_{(\dot H^s)^2}^2,
\end{align}
and
\begin{align} \label{est-til-B2'}
&\norm{\ve B_2^Q(u,v)-\ve B_2^Q(\tilde u,\tilde v)}_{(\dot H^s)^2}  \notag\\
&\leq C(s)\frac{1}{N}\norm{(u,v)-(\tilde u,\tilde v)}_{(\dot H^s)^2}
\left(\norm{(u,v)}_{(\dot H^s)^2}+\norm{(\tilde u,\tilde v)}_{(\dot H^s)^2}\right)    .
\end{align}
\end{lemma}

\begin{proof}
 Observe that every term in $\ve B_2^Q(u,v)$ contains the operator $\mathcal Q$ (projection on high frequencies $|k|>N$), which is the reason that $1/N$ appears in the estimates (\ref{est-til-B2}) and (\ref{est-til-B2'}).
To see this, let us consider a typical term, say, $B_2(\mathcal Pu,\mathcal Qv)$. The rest terms can be estimated similarly. We let $z$ be an element in $\dot H^{-s}$. Consider
\begin{align*}
|(B_2(\mathcal Pu,\mathcal Qv),z)|\leq \frac{1}{6}\sum_{k\in \ZZ_0} \sum_{k_1+k_2=k} \frac{|\mathcal Pu_{k_1}||\mathcal Qv_{k_2}||z_k|}{|k_1||k_2|}
=\frac{1}{6}\sum_{0<|k_1|\leq N} \sum_{|k_2|>N} \frac{|u_{k_1}||v_{k_2}||z_{k_1+k_2}|}{|k_1||k_2|}.
\end{align*}

Set $U_k=|u_k||k|^{-\alpha}$, $V_k=|v_k||k|^s$, $Z_k=|z_k||k|^{-s}$,
where $s\geq 0$, $0\leq \alpha<1/2$. Then
\begin{align*}
|\langle B_2(\mathcal Pu,\mathcal Qv),z \rangle |
&\leq \frac{1}{6}\sum_{0<|k_1|\leq N} \sum_{|k_2|>N} \frac{|U_{k_1}||V_{k_2}||Z_{k_1+k_2}||k_1+k_2|^s}{|k_1|^{1-\alpha}|k_2|^{1+s}} \notag\\
&\leq \frac{1}{6}\sum_{0<|k_1|\leq N} \sum_{|k_2|>N} \frac{|U_{k_1}||V_{k_2}||Z_{k_1+k_2}||2k_2|^s}{|k_1|^{1-\alpha}|k_2|^{1+s}} \notag\\
&\leq C(s)\sum_{0<|k_1|\leq N} \sum_{|k_2|>N} \frac{|U_{k_1}||V_{k_2}||Z_{k_1+k_2}|}{|k_1|^{1-\alpha}|k_2|} \notag\\
&\leq C(s)\frac{1}{N} \sum_{0<|k_1|\leq N} \frac{|U_{k_1}|}{|k_1|^{1-\alpha}} \sum_{|k_2|>N} |V_{k_2}||Z_{k_1+k_2}|\notag\\
&\leq C(s)\frac{1}{N} \left( \sum_{k\in \ZZ_0} \frac{1}{k^{2-2\alpha}}\right)^{\frac{1}{2}} \norm{U}_{\dot H^0} \norm{V}_{\dot H^0} \norm{Z}_{\dot H^0} \notag\\
&\leq C(s,\alpha)\frac{1}{N} \norm{u}_{\dot H^{-\alpha}} \norm{v}_{\dot H^s} \norm{z}_{\dot H^{-s}}.
\end{align*}
By duality, this implies
\begin{align*}
\norm{B_2(\mathcal Pu,\mathcal Qv)}_{\dot H^s} \leq C(s,\alpha)\frac{1}{N} \norm{u}_{\dot H^{-\alpha}} \norm{v}_{\dot H^s}, \text{\;\;for\;\;} s\geq 0, \;0\leq \alpha<1/2.
\end{align*}
Furthermore, by the bilinearity of $B_2(\phi,\psi)$, it is easy to see that
\begin{align*}
&\norm{B_2(\mathcal Pu,\mathcal Qv)-B_2(\mathcal P\tilde u,\mathcal Q\tilde v)}_{\dot H^s} \notag\\
&\leq C(s,\alpha)\frac{1}{N}(\norm{u-\tilde u}_{\dot H^{-\alpha}} \norm{v}_{\dot H^s}+\norm{\tilde u}_{\dot H^{-\alpha}} \norm{v-\tilde v}_{\dot H^s}).
\end{align*}
Obviously $\norm{\phi}_{\dot H^{-\alpha}}\leq \norm{\phi}_{\dot H^{s}}$ for $s\geq 0$ and $\alpha \in [0,1/2)$, so (\ref{est-til-B2}) and (\ref{est-til-B2'}) hold.
\end{proof}

Now, with the mapping properties of $\ve B_1^P$, $\ve B_2^Q$ and $\ve R_3^Q$ discussed above, we prove the uniqueness of solutions and continuous dependence on initial data, which are stated in Theorem \ref{exist}, for $s>1/2$. The less regular case $s\in [0,1/2]$ will be considered in the next section.

\begin{proof}
Integrating the modified first form (\ref{m-1st-form}) gives us
\begin{align} \label{m-int-1st}
\vect{u}{v}(t)-\vect{u}{v}(0) =&\ve B_2^Q(u(t),v(t))-\ve B_2^Q(u(0),v(0)) \notag\\
&+\int_0^t \left[\ve B_1^P(u(\t),v(\t))+ \ve R_3^Q(u(\t),v(\t)) \right]d\t.
\end{align}

Let $(y(t),z(t)):=(u(t),v(t))-(u^{\text{in}},v^{\text{in}})$. In terms of the new variables, (\ref{m-int-1st}) reads
\begin{align} \label{fix-eqn}
(y,z)=\mathcal F(y,z)
\end{align}
where
\begin{align} \label{def-F}
\mathcal F&(y,z)(t):=\ve B_2^Q(y(t)+u^{\text{in}},z(t)+v^{\text{in}})-\ve B_2^Q(u^{\text{in}},v^{\text{in}})  \notag\\
&+\int_0^t \left[\ve B_1^P(y(\t)+u^{\text{in}},z(\t)+v^{\text{in}})+\ve R_3^Q(y(\t)+u^{\text{in}},z(\t)+v^{\text{in}}) \right]d\t.
\end{align}

For $T^\ast>0$ which will be chosen later, consider the Banach space
$$C_0([0,T^*];(\dot H^s)^2):=\{(y,z)\in C([0,T^*];(\dot H^s)^2): (y(0),z(0))=0 \}.$$
We aim to show that the nonlinear operator $\mathcal F$ maps the ball of radius $A$, which is,
\begin{align} \label{ball}
\{(y,z)\in C_0([0,T^*];(\dot H^s)^2): \norm{(y,z)}_{C([0,T^*];(\dot H^s)^2)}\leq A \},
\end{align}
into itself, and it is a contraction map provided that $T^*$ is sufficiently small.

Let $(y,z)$ and $(\tilde y,\tilde z)$ be in the ball (\ref{ball}), then by (\ref{est-til-R3}), (\ref{est-til-R3'}), (\ref{est-low-B1'}), (\ref{est-low-B1''}), (\ref{est-til-B2}), (\ref{est-til-B2'}), and the definition (\ref{def-F}) of $\mathcal F$, we find
\begin{align*}
&\norm{\mathcal F(y,z)(t)}_{(\dot H^s)^2}\leq C(s)\frac{1}{N}  \left(A^2+\norm{(u^{\text{in}},v^{\text{in}})}_{(\dot H^s)^2}^2\right) \notag\\
&+T^*\left[C(s,N)\left(A^2+\norm{(u^{\text{in}},v^{\text{in}})}_{(\dot H^0)^2}^2\right)+C(s)\left(A^3+\norm{(u^{\text{in}},v^{\text{in}})}_{(\dot H^s)^2}^3  \right) \right].
\end{align*}
Notice that the left-hand side, i.e., $\norm{\mathcal F(y,z)(t)}_{(\dot H^s)^2}$ is independent of $N$.
Moreover, we also have
\begin{align*}
&\norm{\mathcal F(y,z)(t)-\mathcal F(\tilde y,\tilde z)(t)}_{(\dot H^s)^2}    \notag\\
&\leq  \norm{(y,z)(t)-(\tilde y,\tilde z)(t)}_{(\dot H^s)^2}\Big\{
C(s)\frac{1}{N}  \left(A+\norm{(u^{\text{in}},v^{\text{in}})}_{(\dot H^s)^2}\right) \notag\\
&+T^*\left[C(s,N)\left(A+\norm{(u^{\text{in}},v^{\text{in}})}_{(\dot H^0)^2}\right)+C(s)\left(A^2+\norm{(u^{\text{in}},v^{\text{in}})}_{(\dot H^s)^2}^2  \right) \right] \Big\},
\end{align*}
for all $t\in [0,T^*]$. We observe once again that the left-hand side of the above inequality does not depend on $N$.
Therefore, for any $A>0$, we can choose $N$ sufficiently large and $T^*$ small enough so that
$\norm{\mathcal F(y,z)(t)}_{(\dot H^s)^2} \leq A$, $\norm{\mathcal F(\tilde y,\tilde z)(t)}_{(\dot H^s)^2} \leq A$,
and $$\norm{\mathcal F(y,z)(t)-\mathcal F(\tilde y,\tilde z)(t)}_{(\dot H^s)^2} \leq \frac{1}{2}\norm{(y,z)(t)-(\tilde y,\tilde z)(t)}_{(\dot H^s)^2},$$
for all $t\in [0,T^*]$, where $T^*$ depends on $A$ and $\norm{(u^{\text{in}},v^{\text{in}})}_{(\dot H^s)^2}$. By the Banach's Fixed Point Theorem, there exists a unique solution $(y,z)$ of (\ref{fix-eqn}) on $[0,T^*]$ in the ball (\ref{ball}), which immediately implies the local existence and uniqueness for the integrated modified first form (\ref{m-int-1st}) in the space $C([0,T^*];(\dot H^s)^2)$, for $s>1/2$.

It can be shown, by elementary analysis, that any solution of the original cKdV (\ref{eq:9}) in the sense of Definition \ref{dfn-solution} also satisfies the integrated modified first form (\ref{m-int-1st}). Therefore, the uniqueness of solutions to (\ref{m-int-1st}) on $[0,T^*]$ implies the uniqueness for (\ref{eq:9}) on $[0,T^*]$. By extension, the global solution constructed in Section \ref{S-4} is the unique solution of (\ref{eq:9}), and it is in the space $C([0,T];(\dot H^s)^2)$, for any $T>0$, and $s>1/2$.

It remains to prove the continuous dependence on initial data. Let $T>0$ be given. We take two different solutions $(u,v)$ and $(\tilde u,\tilde v)$ evolving from two initial points $(u^{\text{in}},v^{\text{in}})$ and $(\tilde u^{\text{in}},\tilde v^{\text{in}})$. Thus, by (\ref{m-int-1st})
\begin{align} \label{dependence}
\vect{u-\tilde u}{v-\tilde v}&(t)
=\vect{u^{\text{in}}-\tilde u^{\text{in}}}{v^{\text{in}}-\tilde v^{\text{in}}} +\ve B_2^Q(u(t),v(t))-\ve B_2^Q(\tilde u(t),\tilde v(t)) \notag\\
&\hspace{1 in }-(\ve B_2^Q(u^{\text{in}},v^{\text{in}})- \ve B_2^Q(\tilde u^{\text{in}},\tilde v^{\text{in}}))\notag\\
&+\int_0^t \left[\ve B_1^P(u,v)-\ve B_1^P(\tilde u,\tilde v)+\ve R_3^Q(u,v)-\ve R_3^Q(\tilde u,\tilde v) \right]d\t.
\end{align}
Due to (\ref{gal31}), which has been proved in Section \ref{S-4}, there exists $M>0$ such that
\begin{align*}
\norm{(u,v)}_{L^{\infty}([0,T];(\dot H^s)^2)} \text{\;\;and\;\;} \norm{(\tilde u,\tilde v)}_{L^{\infty}([0,T];(\dot H^s)^2)}\leq M,
\end{align*}
where $M$ depends on $T$, $s$, $\max\left\{\norm{(u^{\text{in}},v^{\text{in}})}_{(\dot H^0)^2}, \norm{(\tilde u^{\text{in}},\tilde v^{\text{in}})}_{(\dot H^0)^2}\right\}$. Thus,
by taking the $(\dot H^s)^2$-norm on both sides of (\ref{dependence}), and using (\ref{est-til-R3'}), (\ref{est-low-B1''}) and (\ref{est-til-B2'}), we deduce for $t\in [0,T^*]$,
\begin{align*}
&\norm{(u,v)-(\tilde u,\tilde v)}_{L^{\infty}([0,T^*];(\dot H^s)^2)} \leq \norm{(u^{\text{in}},v^{\text{in}})-(\tilde u^{\text{in}},\tilde v^{\text{in}})}_{(\dot H^s)^2} \notag\\
&+C(s)\frac{1}{N} M(\norm{(u,v)-(\tilde u,\tilde v)}_{L^{\infty}([0,T^*];(\dot H^s)^2)}+
\norm{(u^{\text{in}},v^{\text{in}})-(\tilde u^{\text{in}},\tilde v^{\text{in}})}_{(\dot H^s)^2}) \notag\\
&+C(s,N)T^*\norm{(u,v)-(\tilde u,\tilde v)}_{L^{\infty}([0,T^*];(\dot H^s)^2)}(M+M^2).
\end{align*}
Therefore, if we choose $N$ large enough such that $C(s)\frac{1}{N}M\leq \frac{1}{3}$, and $T^*$ sufficient small, such that $C(s,N)T^*(M+M^2)\leq \frac{1}{3}$, then
\begin{align*}
\norm{(u,v)-(\tilde u,\tilde v)}_{L^{\infty}([0,T^*];(\dot H^s)^2)}\leq 4\norm{(u^{\text{in}},v^{\text{in}})-(\tilde u^{\text{in}},\tilde v^{\text{in}})}_{(\dot H^s)^2}.
\end{align*}
By iterating the above procedure $[T/T^*]+1$ times, we obtain
\begin{align*}
\norm{(u,v)-(\tilde u,\tilde v)}_{L^{\infty}([0,T];(\dot H^s)^2)} \leq 4^{[T/T^*]+1}\norm{(u^{\text{in}},v^{\text{in}})-(\tilde u^{\text{in}},\tilde v^{\text{in}})}_{(\dot H^s)^2},
\end{align*}
where $T^*$ depends on $T$, $s$, and $\max\left\{\norm{(u^{\text{in}},v^{\text{in}})}_{(\dot H^0)^2}, \norm{(\tilde u^{\text{in}},\tilde v^{\text{in}})}_{(\dot H^0)^2}\right\}$.
\end{proof}

\smallskip

\section{Uniqueness for $s\in [0,1/2]$} \label{irregular}

 Notice that the mapping property (\ref{est-til-R3}) of $\ve R_3^Q$ holds for $s>1/2$ only. In order to prove the uniqueness for the case $s\in [0,1/2]$, we shall perform integration by parts procedure to $\ve R_3^Q(u,v)$ to obtain operators with nicer mapping properties in $\dot H^s$, for $s\in [0,1/2]$. On the other hand, for the purpose of constructing a contraction mapping, our strategy is similar to the one used in the previous section, that is, decomposing $\ve R_3^Q(u,v)$ appropriately according to high and low Fourier modes so as to take advantage of the time-averaging induced squeezing.

Recall, all terms in $\ve R_3^Q(u,v)_k$ are in the form of (\ref{u-3}). As in Section \ref{sec-dbp}, we single out the resonant terms (i.e. when $(k_1+k_2)(k_2+k_3)(k_3+k_1)=0$) in $\ve R_3^Q(u,v)_k$ by splitting
\begin{align} \label{split-1}
\ve R_3^Q(u,v)=\ve R_{\text{3res}}^Q(u,v)+\ve R_{\text{3nres}}^Q(u,v).
\end{align}

It is easy to see that the resonance $\ve R_{\text{3res}}^Q(u,v)$ has the same mapping property as $\ve R_{\text{3res}}(u,v)$, that is, for $s\geq 0$,
\begin{align} \label{est-R3res}
\norm{\ve R_{\text{3res}}^Q(u,v)}_{(\dot H^s)^2}\leq C\norm{(u,v)}_{(\dot H^0)^2}^2\norm{(u,v)}_{(\dot H^s)^2},
\end{align}
and
\begin{align} \label{est-R3res'}
&\norm{\ve R_{\text{3res}}^Q(u,v)- \ve R_{\text{3res}}^Q(\tilde u,\tilde v) }_{(\dot H^s)^2} \notag\\
&\leq C\norm{(u,v)-(\tilde u,\tilde v)}_{(\dot H^s)^2}
\left(\norm{(u,v)}_{(\dot H^s)^2}^2+\norm{(\tilde u,\tilde v)}_{(\dot H^s)^2}^2\right).
\end{align}

Next, we decompose $\ve R_{\text{3nres}}^Q(u,v)$ by appropriately splitting the Fourier modes of $u$ and $v$ into high and low modes. By (\ref{u-3}), all terms in
$\ve R_{\text{3nres}}^Q(u,v)_k$ can be expressed in the structure
\begin{align} \label{u-4}
\pm \frac{i}{12} \sum_{k_1+k_2+k_3=k \atop \{k_1,k_2,k_3\}\in \mathcal D}^{\text{nonres}} \frac{e^{3i(k_1+k_2)(k_2+k_3)(k_3+k_1)t}}{k_1} \phi_{k_1} \psi_{k_2} \xi_{k_3}, \;\;k\in \ZZ_0,
\end{align}
where $\mathcal D\subset \ZZ_0^3$, and each of $\phi$, $\psi$, $\xi$ is either $u$ or $v$.
Since the explicit structure of $\mathcal D$ is irrelevant to the following argument, we see that (\ref{u-4}) is essentially the same as
$R_{\text{3nres}}(\phi,\psi,\xi)_k$ defined in (\ref{def-nres}), which can be split into two parts
by adopting the idea in \cite{B-I-T}:
\begin{align} \label{decompose}
R_{\text{3nres}}(\phi,\psi,\xi)_k=R_{\text{3nres0}}(\phi,\psi,\xi)_k+R_{\text{3nres1}}(\phi,\psi,\xi)_k
\end{align}
where
\begin{align} \label{R3nres0}
R_{\text{3nres0}}(\phi,\psi,\xi)_k:&=R_{\text{3nres}}(\phi,\mathcal Q\psi,\mathcal Q\xi)_k \notag\\
&=\sum_{k_1+k_2+k_3=k}^{\text{nonres}} \frac{e^{3i(k_1+k_2)(k_2+k_3)(k_3+k_1)t}}{k_1} \phi_{k_1} \mathcal Q\psi_{k_2} \mathcal Q\xi_{k_3}
\end{align}
and
\begin{align} \label{R3nres1}
&R_{\text{3nres1}}(\phi,\psi,\xi)_k \notag\\ &:=R_{\text{3nres}}(\phi,\psi,\xi)_k-R_{\text{3nres}}(\phi,\mathcal Q\psi,\mathcal Q\xi)_k \notag\\
&=R_{\text{3nres}}(\phi,\mathcal P\psi,\xi)_k+R_{\text{3nres}}(\phi,\mathcal Q\psi,\mathcal P\xi)_k \notag\\
&=\sum_{k_1+k_2+k_3=k}^{\text{nonres}} \frac{e^{3i(k_1+k_2)(k_2+k_3)(k_3+k_1)t}}{k_1}
(\phi_{k_1} \mathcal P\psi_{k_2} \xi_{k_3}+\phi_{k_1} \mathcal Q\psi_{k_2} \mathcal P\xi_{k_3}).
\end{align}
It has been remarked in \cite{B-I-T} that, $R_{\text{3nres0}}(\phi,\psi,\xi)_k$ has only one smoothed factor $\frac{\phi_{k_1}}{k_1}$.
However, every term in $R_{\text{3nres1}}(\phi,\psi,\xi)_k$ has two smoothed factors: $\frac{\phi_{k_1}}{k_1}$ and either $\mathcal P\psi_{k_2}$ or $\mathcal P\xi_{k_3}$.
The following mapping property of $R_{\text{3nres1}}(\phi,\psi,\xi)$ is a special case of
Lemma \ref{L:R3nres1}:
\begin{align} \label{u-5}
\norm{R_{\text{3nres1}}(\phi,\psi,\xi)}_{\dot H^s}
\leq C N^{s+1}\norm{\phi}_{\dot H^0}\norm{\psi}_{\dot H^0}\norm{\xi}_{\dot H^0}+C N \norm{\phi}_{\dot H^0}\norm{\psi}_{\dot H^0 }\norm{\xi}_{\dot H^s},
\end{align}
for $s\in [0,1]$.

We can decompose every term in $\ve R_{\text{3nres}}^Q(u,v)_k$ as (\ref{decompose}), and it follows that
\begin{align} \label{split-2}
\ve R_{\text{3nres}}^Q(u,v)_k=\ve R_{\text{3nres0}}^Q(u,v)_k+\ve R_{\text{3nres1}}^Q(u,v)_k,
\end{align}
where all terms in $\ve R_{\text{3nres0}}^Q(u,v)_k$ are in the form
\begin{align} \label{u-6}
\pm \frac{i}{12} \sum_{k_1+k_2+k_3=k \atop \{k_1,k_2,k_3\}\in \mathcal D}^{\text{nonres}} \frac{e^{3i(k_1+k_2)(k_2+k_3)(k_3+k_1)t}}{k_1} \phi_{k_1} \mathcal Q\psi_{k_2} \mathcal Q\xi_{k_3}, \;\;k\in \ZZ_0,
\end{align}
while all terms in $\ve R_{\text{3nres1}}^Q(u,v)_k$ have the structure
\begin{align*}
\pm \frac{i}{12} \sum_{k_1+k_2+k_3=k \atop \{k_1,k_2,k_3\}\in \mathcal D}^{\text{nonres}} \frac{e^{3i(k_1+k_2)(k_2+k_3)(k_3+k_1)t}}{k_1} (\phi_{k_1} \mathcal P\psi_{k_2} \xi_{k_3}+\phi_{k_1} \mathcal Q\psi_{k_2} \mathcal P\xi_{k_3}), \;\;k\in \ZZ_0,
\end{align*}
where $\mathcal D\subset \ZZ_0^3$, and $\phi$, $\psi$, $\xi$ is either $u$ or $v$. By (\ref{u-5}) we infer, for $0\leq s \leq 1$,
\begin{align} \label{u-17}
\norm{\ve R_{\text{3nres1}}^Q(u,v)}_{(\dot H^s)^2}\leq C(N,s)
\norm{(u,v)}_{(\dot H^s)^2} \norm{(u,v)}^2_{(\dot H^0)^2},
\end{align}
and
\begin{align} \label{u-17'}
&\norm{\ve R_{\text{3nres1}}^Q(u,v)-\ve R_{\text{3nres1}}^Q(\tilde u,\tilde v)}_{(\dot H^s)^2} \notag\\
&\leq C(N,s) \norm{(u,v)-(\tilde u,\tilde v)}_{(\dot H^s)^2}
\left(\norm{(u,v)}_{(\dot H^s)^2}^2+\norm{(\tilde u,\tilde v)}_{(\dot H^s)^2}^2\right),
\end{align}
where $C(N,s)\rightarrow \infty$ as $N\rightarrow \infty$.

Now, we apply the differentiation by parts to $\ve R_{\text{3nres0}}(u,v)$. Note all terms in $\ve R_{\text{3nres0}}(u,v)_k$ are in the form (\ref{u-6}), and we can take the following term as an example:
\begin{align} \label{u-7}
-\frac{i}{12} \sum_{k_1+k_2+k_3=k \atop \{k_1,k_2,k_3\}\in \mathcal D}^{\text{nonres}} \frac{e^{3i(k_1+k_2)(k_2+k_3)(k_3+k_1)t}}{k_1} u_{k_1} \mathcal Qu_{k_2} \mathcal Qv_{k_3},
\end{align}
where $\mathcal D=\{\{k_1,k_2,k_3\} \in \ZZ_0^3: |k_1|\leq N, \; |k_2+k_3|> N\}$, which is corresponding to the first term of $f_k$ defined in (\ref{u-2}). If we ignore the explicit structure of the set $\mathcal D$, which is irrelevant to our following argument, then (\ref{u-7}) is essentially the same as $R_{\text{3nres0}}(u,u,v)_k$, defined in (\ref{R3nres0}), to which we carry out the differentiation by parts:
\begin{align} \label{u-8}
&R_{\text{3nres0}}(u,u,v)_k \notag\\
&:=\sum_{k_1+k_2+k_3=k}^{\text{nonres}} \frac{e^{3i(k_1+k_2)(k_2+k_3)(k_3+k_1)t}}{k_1} u_{k_1} \mathcal Q u_{k_2} \mathcal Q v_{k_3} \notag\\
&=\frac{1}{3i}\left(\partial_t B_{30}(u,u,v)_k-g_k\right),
\end{align}
where
\begin{equation}
  \label{B30}
  B_{30}(\phi,\psi,\xi)_k:=
\sum_{k_1+k_2+k_3=k}^{\text{nonres}}
\frac{e^{3i(k_1+k_2)(k_2+k_3)(k_1+k_3)t}}{k_1(k_1+k_2)(k_2+k_3)(k_1+k_3)}
\phi_{k_1}\mathcal Q\psi_{k_2} \mathcal Q \xi_{k_3} \:,
\end{equation}
and
\begin{align}  \label{g-k}
g_k:=\sum_{k_1+k_2+k_3=k}^{\text{nonres}} \frac{e^{3i(k_1+k_2)(k_2+k_3)(k_3+k_1)t}}{k_1(k_1+k_2)(k_2+k_3)(k_3+k_1)}
\partial_t (u_{k_1} \mathcal Qu_{k_2} \mathcal Qv_{k_3}).
\end{align}

An analogue to (\ref{u-8}), one can complete the differentiation by parts procedure to all terms in $\ve R_{\text{3nres0}}(u,v)$. Hence
\begin{align} \label{2-dif-part}
\ve R_{\text{3nres0}}(u,v)_k=\partial_t \ve B_{30}(u,v)_k+\ve B_{40}(u,v)_k.
\end{align}
For the sake of conciseness, we do no provide the exact formulas of $\ve B_{30}(u,v)$ and $\ve B_{40}(u,v)$. But notice that, $\frac{1}{3i}B_{30}(u, u, v)_k$ is a typical term in $\ve B_{30}(u,v)_k$, so by virtue of Lemma \ref{L:B30} one has, for $0\leq s\leq 1$,
\begin{align}  \label{u-13}
\norm{\ve B_{30}(u,v)}_{(\dot H^s)^2}\leq \gamma(N,s)\norm{(u,v)}_{(\dot H^s)^2}^3,
\end{align}
and
\begin{align}  \label{u-13'}
&\norm{\ve B_{30}(u,v)-\ve B_{30}(\tilde u,\tilde v)}_{(\dot H^s)^2} \notag\\
&\leq  \gamma(N,s)  \norm{(u,v)-(\tilde u,\tilde v)}_{(\dot H^s)^2}
\left(\norm{(u,v)}_{(\dot H^s)^2}^2+\norm{(\tilde u,\tilde v)}_{(\dot H^s)^2}^2\right),
\end{align}
where $\gamma(N,s)\rightarrow 0$ as $N\rightarrow \infty$.
In addition, $g_k$ defined in (\ref{g-k}) is a typical term in $\ve B_{40}(u,v)_k$.
Clearly, $g_k$ can be treated in the same way as (\ref{star}), and generates terms in structures of $B_4^1(\phi,\psi,\xi,\eta)$ or $B_4^2(\phi,\psi,\xi,\eta)$. By means of the mapping property of $B_4$ provided in Lemma \ref{L:B4}, one has, for $s\geq 0$,
\begin{align}  \label{u-14}
\norm{\ve B_{40}(u,v)}_{(\dot H^s)^2}\leq C(s) \norm{(u,v)}_{(\dot H^s)^2}^4,
\end{align}
and
\begin{align}  \label{u-14'}
&\norm{\ve B_{40}(u,v)-\ve B_{40}(\tilde u,\tilde v)}_{(\dot H^s)^2} \notag\\
&\leq C(s) \norm{(u,v)-(\tilde u,\tilde v)}_{(\dot H^s)^2}
\left(\norm{(u,v)}_{(\dot H^s)^2}^3+\norm{(\tilde u,\tilde v)}_{(\dot H^s)^2}^3\right).
\end{align}

By virtue of (\ref{split-1}), (\ref{split-2}) and (\ref{2-dif-part}), we can write $\ve R_3^Q(u,v)$ as
\begin{align} \label{u-15}
\ve R_3^Q(u,v)&=\ve R_{\text{3res}}^Q(u,v)+\ve R_{\text{3nres}}^Q(u,v) \notag\\
&=\ve R_{\text{3res}}^Q(u,v)+\ve R_{\text{3nres1}}^Q(u,v)+\ve R_{\text{3nres0}}^Q(u,v) \notag\\
&=\ve R_{\text{3res}}^Q(u,v)+\ve R_{\text{3nres1}}^Q(u,v)+\partial_t \ve B_{30}(u,v)+\ve B_{40}(u,v).
\end{align}
Substituting (\ref{u-15}) into the modified first form (\ref{m-1st-form}), we obtain the following \emph{modified second form of the cKdV}:
\begin{align} \label{m-2nd-form}
&\partial_t \left[\vect{u}{v}-\ve B_2^Q(u,v)-\ve B_{30}(u,v) \right]  \notag\\
&=\ve B_1^P(u,v)+ \ve R_{\text{3res}}^Q(u,v)+\ve R_{\text{3nres1}}^Q(u,v)+\ve B_{40}(u,v).
\end{align}

 We now use this form of the cKdV to prove the uniqueness of global solutions and continuous dependence on initial data in the space $(\dot H^s)^2$ for $s\in (0,1/2]$.

\begin{proof}
The integrated form of (\ref{m-2nd-form}) reads
\begin{align}  \label{u-16}
\vect{u}{v}(t)-&\vect{u}{v}(0)=\left[\ve B_2^Q(u,v)+\ve B_{30}(u,v)\right](t)
-\left[\ve B_2^Q(u,v)+\ve B_{30}(u,v)\right](0) \notag\\
&+\int_0^t  \left[\ve B_1^P(u,v)+ \ve R_{\text{3res}}^Q(u,v)+\ve R_{\text{3nres1}}^Q(u,v)+\ve B_{40}(u,v)\right] (\t) d\t.
\end{align}
Let $(y(t),z(t)):=(u(t),v(t))-(u^{\text{in}},v^{\text{in}})$. Using the new variables $y$ and $z$, (\ref{u-16}) can be written as a fixed point equation
\begin{align} \label{fix-eqn-2}
(y,z)=\mathscr F(y,z)
\end{align}
where
\begin{align} \label{def-F2}
\mathscr F&(y,z)(t)=\ve B_2^Q(y(t)+u^{\text{in}},z(t)+v^{\text{in}})+\ve B_{30}(y(t)+u^{\text{in}},z(t)+v^{\text{in}}) \notag\\
&\hspace{1 in}-\ve B_2^Q(u^{\text{in}},v^{\text{in}})-\ve B_{30}(u^{\text{in}},v^{\text{in}})\notag\\
&+\int_0^t \Big[\ve B_1^P(y(\t)+u^{\text{in}},z(\t)+v^{\text{in}})+ \ve R_{\text{3res}}^Q(y(\t)+u^{\text{in}},z(\t)+v^{\text{in}}) \notag\\
&\hspace{0.5 in}+\ve R_{\text{3nres1}}^Q(y(\t)+u^{\text{in}},z(\t)+v^{\text{in}})+\ve B_{40}(y(\t)+u^{\text{in}},z(\t)+v^{\text{in}})\Big] d\t.
\end{align}

We intend to show $\mathscr F$ maps the ball of radius $A$, which is,
\begin{align} \label{ball-1}
\{(y,z)\in C_0([0,T^*];(\dot H^s)^2): \norm{(y,z)}_{C([0,T^*];(\dot H^s)^2)}\leq A \},
\end{align}
for $s\in [0,1/2]$, into itself and is a contraction map provided that $T^*$ is sufficiently small. To see this,
we let $(y,z)$ and $(\tilde y,\tilde z)$ be in the ball (\ref{ball-1}). Then, due to the mapping properties (\ref{est-low-B1'}), (\ref{est-low-B1''}), (\ref{est-til-B2}), (\ref{est-til-B2'}),  (\ref{est-R3res}), (\ref{est-R3res'}), (\ref{u-17}), (\ref{u-17'}), (\ref{u-13}), (\ref{u-13'}),
(\ref{u-14}), (\ref{u-14'}) and the definition (\ref{def-F2}) of $\mathscr F$, we deduce that, for $s\in [0,1/2]$,
\begin{align*}
&\norm{\mathscr F(y,z)(t)}_{(\dot H^s)^2}\leq \frac{C(s)}{N}\left(A^2+\norm{(u^{\text{in}},v^{\text{in}})}_{(\dot H^s)^2}^2\right)+\gamma(N,s)\left(A^3+\norm{(u^{\text{in}},v^{\text{in}})}_{(\dot H^s)^2}^3\right)\notag\\
&+T^*\Big[C(s,N)\left(A^2+\norm{(u^{\text{in}},v^{\text{in}})}_{(\dot H^s)^2}^2\right)
+C\left(A^3+\norm{(u^{\text{in}},v^{\text{in}})}_{(\dot H^s)^2}^3\right) \notag\\
&\hspace{0.4 in}+C(N,s)\left(A^3+\norm{(u^{\text{in}},v^{\text{in}})}_{(\dot H^s)^2}^3\right)
+C(s)\left(A^4+\norm{(u^{\text{in}},v^{\text{in}})}_{(\dot H^s)^2}^4\right)\Big].
\end{align*}
One observes that the left-hand side, i.e., $\norm{\mathscr F(y,z)(t)}_{(\dot H^s)^2}$ is independent of $N$. In addition,
\begin{align*}
&\norm{\mathscr F(y,z)(t)-\mathscr F(\tilde y,\tilde z)(t)}_{(\dot H^s)^2}
\leq  \norm{(y,z)(t)-(\tilde y,\tilde z)(t)}_{(\dot H^s)^2}                 \notag\\
&\times \Big\{\frac{C(s)}{N}\left(A+\norm{(u^{\text{in}},v^{\text{in}})}_{(\dot H^s)^2}\right)+\gamma(N,s)\left(A^2+\norm{(u^{\text{in}},v^{\text{in}})}_{(\dot H^s)^2}^2\right)\notag\\
&\hspace{0.2 in}+T^*\Big[C(s,N)\left(A+\norm{(u^{\text{in}},v^{\text{in}})}_{(\dot H^s)^2}\right)
+C\left(A^2+\norm{(u^{\text{in}},v^{\text{in}})}_{(\dot H^s)^2}^2\right) \notag\\
&\hspace{0.4 in}+C(N,s)\left(A^2+\norm{(u^{\text{in}},v^{\text{in}})}_{(\dot H^s)^2}^2\right)
+C(s)\left(A^3+\norm{(u^{\text{in}},v^{\text{in}})}_{(\dot H^s)^2}^3\right)\Big]\Big\},
\end{align*}
where $\gamma(N,s)\rightarrow 0$ as $N\rightarrow \infty$. We observe once again that the left-hand side of the above inequality is independent of $N$.
Thus, for any $A>0$, we can choose $N$ sufficiently large, and $T^*$ small enough, so that
$\norm{\mathscr F(y,z)(t)}_{(\dot H^s)^2} \leq A$, $\norm{\mathscr F(\tilde y,\tilde z)(t)}_{(\dot H^s)^2} \leq A$,
and $$\norm{\mathscr F(y,z)(t)-\mathscr F(\tilde y,\tilde z)(t)}_{(\dot H^s)^2} \leq \frac{1}{2}\norm{(y,z)(t)-(\tilde y,\tilde z)(t)}_{(\dot H^s)^2},$$
for all $t\in [0,T^*]$, where $T^*$ depends on $A$ and $\norm{(u^{\text{in}},v^{\text{in}})}_{(\dot H^s)^2}$. By the Banach's Fixed Point Theorem, there exists a unique solution $(y,z)$ of (\ref{fix-eqn-2}) on $[0,T^*]$ in the ball (\ref{ball-1}), which yields the short-time existence and uniqueness of the solution $(u,v)$ to (\ref{u-16}) in the space $(\dot H^s)^2$ for $s\in [0,1/2]$.

It can be shown that any solution of the original cKdV (\ref{eq:9}) in the sense of Definition \ref{dfn-solution} also satisfies the integrated modified second form (\ref{u-16}).
Also, recall in Section \ref{S-4}, we have already proved the global existence of solutions for the original cKdV (\ref{eq:9}) for $s\in (0,1/2]$. Therefore, according to the uniqueness result proved above and using some extension argument, we conclude that the global solution of (\ref{eq:9}) is unique, and it is in the space $C([0,T];(\dot H^s)^2)$ for any $T>0$, and $s\in (0,1/2]$.

Finally, similar to the proof in Section \ref{regular}, we can also show the continuous dependence on the initial data for the case $s\in (0,1/2]$.
\end{proof}

\begin{remark}\label{6-3r}
Notice from the above that the contraction mapping argument is valid for $s=0$. Hence,
provided there is a solution of (\ref{eq:9}) for the case $s=0$, we also obtain the uniqueness and continuous dependence on initial data for the equation (\ref{eq:9}) if $s=0$.
\end{remark}

It remains to show the existence of a solution to (\ref{eq:9}) for the case $s=0$. This will be done by using density arguments.

\begin{proof}
We approximate the initial data $(u^{\text{in}},v^{\text{in}})\in (\dot H^0)^2$ by a sequence of smoother functions $(u_j^{\text{in}},v_j^{\text{in}})\in (\dot H^{s})^2$, where $s>0$. Let us fix an arbitrary $T>0$.
We have already shown that, for each $j$, there exists a unique solution
$(u_j(t),v_j(t)) \in C([0,T];(\dot H^0)^2)$ such that
$(u_j(0),v_j(0))=(u_j^{\text{in}},v_j^{\text{in}})$, and the quantity $\mathcal E(u_j(t),v_j(t))$, defined in (\ref{conserve}), is conserved.
By Remark \ref{6-3r}, we infer
\begin{equation}
  \label{eq:59}
  \norm{(u_j,v_j)-(u_{\ell},v_{\ell})}_{L^\infty([0,T],(\dot H^0)^2)}
  \leq  L\norm{(u_j^{\text{in}},v_j^{\text{in}})-(u_{\ell}^{\text{in}},v_{\ell}^{\text{in}})}_{(\dot H^0)^2} \:,
\end{equation}
	where $L$ depends on $T$ and $\norm{(u^{\text{in}},v^{\text{in}})}_{(\dot H^0)^2}$.
Since $(u_j^{\text{in}},v_j^{\text{in}})$ is a Cauchy sequence in
	$(\dot H^0)^2$, we deduce from (\ref{eq:59}) that
	$(u_j,v_j)$ is a Cauchy sequence in $C([0,T];(\dot
	H^0)^2)$, whereas we denote the limit as
	$(u,v)$. Now, using the mapping properties of
	$B_1$, one can pass to the limit in equation~(\ref{eq:11}) similar as in Section \ref{S-4} (where we proved the existence for $s>0$), and deduce that $(u,v)$ also
	satisfies~(\ref{eq:11}) and conserves $\mathcal E(u(t),v(t))$, defined in (\ref{conserve}), for all $t\in [0,T]$. Finally, by Remark \ref{6-3r} again, we obtain the desired uniqueness
    and continuous dependence on initial data under the $(\dot H^0)^2$ norm.
\end{proof}

\smallskip

\section{Appendix: relevant estimates}\label{appendix}

	In this section we collect all relevant estimates for the
	nonlinear operators entering in our equations. The notations are
	not, or just slightly, different from the ones used
	in~\cite{B-I-T}, where all proofs can be found.

\begin{lemma}\label{L:B1}
	Let $\theta>3/2$. Then the bilinear operator $B_{1}$ defined
	in (\ref{eq:10a}) maps $\dot{H}^{0}\times\dot{H}^{0}$ into
	$\dot{H}^{-\theta}$ and satisfies the estimate
\begin{equation*}
    \|B_{1}(\phi,\psi)\|_{\dot{H}^{-\theta}}\leq C(\theta)
    \|\phi\|_{\dot{H}^{0}}\|\psi\|_{\dot{H}^{0}}.
\end{equation*}
\end{lemma}

\begin{lemma}\label{L:B2}
	Let $s>-1/2$. Then the bilinear operator $B_{2}$ defined in
	(\ref{eq:13a}) maps $\dot{H}^{s}\times \dot{H}^{s}$ into
	$\dot{H}^{s+1}$ and satisfies the estimate
\begin{equation*}
    \|B_{2}(\phi,\psi)\|_{\dot{H}^{s+1}}\leq C(s)
    \|\phi\|_{\dot{H}^{s}}\|\psi\|_{\dot{H}^{s}}.  \label{estB2}
\end{equation*}
\end{lemma}

\begin{lemma}\label{L:B21}
	Let $s+\alpha \geq 0$, $\alpha <3/4$, $s>-3/4$. Then the
	bilinear operator $B_{2}$ defined in~(\ref{eq:13a}) maps
	$\dot{H}^{s}\times \dot{H}^{s}$ into $\dot{H}^{s+\alpha }$ and
	satisfies the estimate
\begin{equation*}
    \|B_{2}(\phi,\psi)\|_{\dot{H}^{s+\alpha }}\leq C(s,\alpha) \|\phi\| _{\dot{H}^{s}}\|\psi\|_{\dot{H}^{s}}.
\end{equation*}
\end{lemma}

\begin{lemma}\label{L:B3}
	Let $s\geq 0$. Then the trilinear operator $B_{3}$
	defined in~(\ref{dbp10}) maps $(\dot{H}^{s})^3$ into $\dot{H}^{s+2}$
	and satisfies the estimate
\begin{equation*}
\Vert B_{3}(\phi,\psi,\xi)\Vert _{\dot{H}^{s+2}}\leq
c(s) \Vert \phi \Vert _{\dot{H} ^{s}}\Vert \psi \Vert _{\dot{H}^{s}}\Vert
\xi \Vert _{\dot{H}^{s}}.
\end{equation*}
\end{lemma}

\begin{lemma} \label{L:B30}
	If $0<s\le1$, then
\begin{equation*}
    \|B_{30}(u,u,v)\|_{\dot{H}^{s}}+
    \|B_{30}(u,v,u)\|_{\dot{H}^{s}}+
    \|B_{30}(v,u,u)\|_{\dot{H}^{s}}\leq
    \frac{C}{N^{s}}\|u\|_{\dot{H}^{0}}^{2}\|v\|_{\dot{H}^{s}}.
\end{equation*}%
	If $s\le0$ and $p=-s\leq 1$, $\alpha >0$, $p+2\alpha <5/3$, %
	and $\alpha<5/6$, then
\begin{equation*}
    \|B_{30}(u,u,v)\|_{\dot{H}^{s}}+
    \|B_{30}(u,v,u)\|_{\dot{H}^{s}}+
    \|B_{30}(v,u,u)\|_{\dot{H}^{s}}\leq
    \frac{C(p,\alpha)}{N^{2\alpha}}
    \|u\|_{\dot{H}^{0}}^{2}\|v\|_{\dot{H}^{s}}.
\end{equation*}
\end{lemma}

\begin{lemma}\label{L:R3}
	Let $s>1/2$. Then the trilinear operator $R_{3}$ defined in (\ref{eq:13b}) maps $(\dot{H}^{s})^{3}$ into $\dot{H}^{s}$ and satisfies the estimate
\begin{equation*}
    \Vert R_{3}(\phi,\psi,\xi)\Vert _{\dot{H}^{s}}\leq C(s) \Vert \phi\Vert _{\dot{H}%
      ^{s}}\Vert \psi\Vert _{\dot{H}^{s}}\Vert \xi\Vert
    _{\dot{H}^{s}}.
\end{equation*}
\end{lemma}

\begin{lemma} \label{L:R3nres1}
Let $0\leq s \leq 1$, $\a\geq 0$. Then the operator $R_{\emph{3nres1}}$ in (\ref{R3nres1}) satisfies the estimate
\begin{align*}
\norm{R_{\emph{3nres1}}(\phi,\psi,\xi)}_{\dot H^s}
\leq C N^{s+1+\alpha}\norm{\phi}_{\dot H^0}\norm{\psi}_{\dot H^{-\alpha}}\norm{\xi}_{\dot H^0} +C N^{1+\alpha} \norm{\phi}_{\dot H^0}\norm{\psi}_{\dot H^{-\alpha}}\norm{\xi}_{\dot H^s}.
\end{align*}
\end{lemma}

\begin{lemma} \label{L:B4}
	Let $s\geq 0$ and $\epsilon \in (0,\frac{1}{2})$. Then the
	multi-linear operator $B_{4}$ maps $(\dot{H}^{s})^{4}$ into
	$\dot H^{s+\epsilon}$ and satisfies the estimate
\begin{equation*}
    \Vert B_{4}(\phi,\psi,\xi,\eta )\Vert _{\dot{H}^{s+\epsilon }}\leq
    C(s,\epsilon )\Vert \phi\Vert _{\dot{H}^{s}}\Vert \psi\Vert
    _{\dot{H}^{s}}\Vert \xi\Vert _{\dot{H}^{s}}\Vert \eta \Vert
    _{\dot{H}^{s}}.
\end{equation*}
\end{lemma}

\smallskip

\par\smallskip\noindent
\textbf{Acknowledgement\,:}    K. S. and E. S. T. would like to thank the Freie Universit\"at Berlin for the kind hospitality where this work was initiated. This work was supported in part by the Minerva Stiftung/Foundation, and  by the NSF
grants DMS-1009950, DMS-1109640 and DMS-1109645.

\end{document}